\documentclass[reqno]{amsart}  %

\usepackage{amssymb}
\usepackage{tikz}
\usetikzlibrary{%
    shapes.geometric,
    matrix,
    intersections,
    positioning,
    arrows,
    arrows.meta,
    decorations.markings,
    calc
}
\tikzset{
     lface/.style={dotted,     decoration={markings, mark=at position 1.0 with {\arrow[>=angle 60]{>}}}, postaction={decorate}},
   lsquare/.style={very thick, decoration={markings, mark=at position {0.5*\pgfdecoratedpathlength+3.5pt} with {\arrow{open square}}}, postaction={decorate}},
  ldiamond/.style={very thick, decoration={markings, mark=at position {0.5*\pgfdecoratedpathlength+5.5pt} with {\arrow{open diamond}}}, postaction={decorate}},
  lvertex/.style={draw, very thick, circle}
}

\makeatletter
\newcommand\underparen[1]{%
  \mathop{\vtop{\m@th\ialign{##\crcr
        $\hfil\displaystyle{#1}\hfil$\crcr
        \noalign{\kern3\p@\nointerlineskip}$\m@th\setbox\z@\hbox{$\braceld$}%
  \bracelu\leaders\vrule \@height\ht\z@ \@depth\z@\hfill\braceru$\crcr}}}\limits}
\newcommand*{\defeq}{\mathrel{\rlap{%
                     \raisebox{0.3ex}{$\m@th\cdot$}}%
                     \raisebox{-0.3ex}{$\m@th\cdot$}}=}
\makeatother

\newtheorem{theorem}{Theorem}[section]
\newtheorem{proposition}[theorem]{Proposition}

\newtheorem{question}[theorem]{Question}
\newtheorem{example}[theorem]{Example}
\newtheorem{definition}[theorem]{Definition}
\newtheorem{algorithm}[theorem]{Algorithm}
\newtheorem{remark}[theorem]{Remark}

\begin{document}
\title{Eberhard-type theorems with two kinds of polygons}

\author{Sebastian Manecke}

\address{FB 12 - Institut f\"ur Mathematik\\
  Goethe-Universit\"at Frankfurt,
  Robert-Mayer-Str. 10\\
  D-60325 Frankfurt am Main, Germany\\
  E-mail: manecke@math.uni-frankfurt.de}

\begin{abstract}
  Eberhard-type theorems are statements about the realizability of a
  polytope (or more general polyhedral maps) given the valency of its
  vertices and sizes of its polygonal faces up to a linear
  degree of freedom. We present new theorems of Eberhard-type where we
  allow adding two kinds of polygons and one type of vertices. We also
  hint towards a full classification of these types of results.
\end{abstract}


\maketitle
\section{Overview}

The classical Eberhard theorems are two results on the
constructability of $r$-valent $3$-polytopes for a given sequence
$(p_k)_{k \geq 3}$, where $p_k \in \mathbb{N} = \{0, 1, 2, \dots\}$ describes
the number of occurrences of each $k$-gon. For us, a \textbf{sequence}
$a$ will always be a function $\mathbb{N} \setminus \{0, 1, 2\} \to \mathbb{N}$ with
finite support.

The original formulations were (see \cite{ConvexPolytopes}):
\begin{theorem}[Eberhard's theorem for $3$-valent
$3$-polytopes]\label{thm:eberhard:3} Let $(p_3, p_4, p_5, \break p_7, \dots,
p_m)$ be a sequence of natural numbers for which
\begin{align}
  \sum_{3 \leq k \leq m \atop k \neq 6}(6 - k) \cdot p_k = 12,
  \label{eq:eberhard:3}
\end{align}
holds. Then there exists a number $p_6$ and a $3$-valent
$3$-polytope which has exactly $p_k$ $k$-gons for each $3 \leq k \leq
m$.
\end{theorem}
\begin{theorem}[Eberhard's theorem for $4$-valent
$3$-polytopes]\label{thm:eberhard:4} Let $(p_3, p_5, \dots, p_m)$ be a
sequence of natural numbers for which
\begin{align}
  \sum_{3 \leq k \leq m \atop k \neq 4}(4 - k) \cdot p_k = 8,
  \label{eq:eberhard:4}
\end{align}
holds. Then there exists a number $p_4$ and a $4$-valent
$3$-polytope which has exactly $p_k$ $k$-gons for each $3 \leq k \leq
m$.
\end{theorem}

If $p_k$ is the number of $k$-gons of a $3$-valent (resp. $4$-valent)
polytope, then \eqref{eq:eberhard:3} (resp.  \eqref{eq:eberhard:4})
holds as an immediate consequence of Euler's relation and double
counting of the number of edges. Thus these theorems answer the
question under which condition a sequence $p$ that suffices the
natural combinatorial conditions has a realization as a polytope.

In the 1970's Barnette, Ernest, Gr\"unbaum, Jendrol', Jucovi{\v{c}},
Trenkler and Zaks~\cite{jucovivc1973theorem, barnette1971toroidal, grunbaum1969planar, zaks1971analogue, jendrol1977generalization} extended
these results to polyhedral maps, which are graph embeddings on a
closed topological $2$-manifold (or \textbf{surface}) generalizing the
combinatorics of $3$-polytopes. In the two setups described by the
classical Eberhard theorems there is now a complete
characterization for which sequences $(p_k)_{k\geq 3}$ and
$(v_k)_{k \geq 3}$ and surfaces $S$ there exists a polyhedral map on
$S$ with $p_k$ $k$-gons and $v_k$ $k$-valent vertices when choosing
the value of $p_6$ and $v_3$, resp. $p_4$ and $v_4$, appropriately.

We want to call the sequences $p = (p_k)_{k \geq 3}$ and
$v = (v_k)_{k \geq 3}$ which count the number of $k$-gons and
$k$-valent vertices of a polyhedral map $M$ the \textbf{$p$-vector} and
the \textbf{$v$-vector} of $M$ and call the pair $(p, v)$ to be
realizable on a surface $S$, if there exists a polyhedral map $M$ on
$S$ with $p$-vector $p$ and $v$-vector $v$.

An easy construction shows that we can in fact find infinitely many
$p_6$ and $v_3$, resp. $p_4$ and $v_4$, such that $(p, v)$ is
realizable as a polyhedral map on a fixed $S$. By using Euler's
relation and an easy double counting argument one cannot increase
$p_6$ and $v_3$, resp.\ $p_4$, $v_4$ independently from each other and
thus one can deduce that there is a linear relation between these
numbers. We want to propose the generalized Eberhard problem:

\begin{question}\label{quest:gen:eberhard}
  Let $p = (p_k)_{k \geq 3}$, $v = (v_k)_{k \geq 3}$,
  $q = (q_k)_{k \geq 3}$ and $w = (w_k)_{k \geq 3}$ be sequences and
  $S$ be a surface. Does there exist infinitely many $c, d \in \mathbb{N}$ and
  a polyhedral map on $S$ with $p$-vector $p + c \cdot q$
  and $v$-vector $v + d \cdot w$?
\end{question}

Theorem~\ref{thm:eberhard:3} and its generalization to polyhedral maps
answer this question for $q = (0, 0, 0, 1, 0, \dots)$ and
$w = (1, 0, \dots)$, whereas Theorem~\ref{thm:eberhard:4} and its
generalizations answer this question for $q = (0, 1, 0 \dots)$,
$w = (0, 1, 0, \dots)$. It is easy to check, that the only missing
possibility for $q$ and $w$ with exactly one non-zero entry, where
such a statement can be true, is $q = (1, 0, \dots)$,
$w = (0, 0, 0, 1, 0, \dots)$. This case is in fact just the dual of
Theorem~\ref{thm:eberhard:3}, and therefore all cases with exactly one
non-zero entry in both $q$ and $w$ have been classified.

Question~\ref{quest:gen:eberhard} for $q$, $w$ with more than one
non-zero entry was first considered by DeVos et
al.~\cite{devos2010eberhard}, who gave an answer in the case of $v =
(1, 0, \dots)$, $q = (0, 0, 1, 0, 1, 0, \dots)$ and $w = (1, 0,
\dots)$ for any surface. We will also consider similar theorems of
this type in this article. To state them more easily, let us introduce
the following notation: Define $[i]$ to be the sequence $a$ with $a_i
= 1$ and $a_j = 0$ for $i \neq j$. We then set $[a_{k_1} \times k_1,
  a_{k_2} \times k_2, \dots, a_{k_n} \times k_n] \defeq \sum_{i = 1}^n
a_{k_i} [k_i]$, where only entries $a_{k_i}$ not equal to zero
occur. If $a_{k_i} = 1$, we will just write $k_i$ instead of $1 \times
k_i$.

In his master thesis~\cite{manecke_master_thesis}, the author gave a
complete answer to Question~\ref{quest:gen:eberhard} in the case that
$q$ has precisely two non-zero and coprime entries and $w$ has one
non-zero entry. We will state the full theorem in
Sec.~\ref{sec:polymap} and give the ideas for the constructions used
in the proofs in Sec.~\ref{sec:construction}. The last Section,
Sec.~\ref{sec:3:4}, will show how these constructions yield one case
of the full statement in \cite{manecke_master_thesis}, that is
to say, the following two theorems:

\begin{theorem}\label{eberhard:3:5:4}
  Let $p$ and $v$ be a pair of admissible sequences for an orientable
  closed $2$-manifold $S$ and $k \in \mathbb{N}$. Then there exists infinitely
  many $c, d \in \mathbb{N}$ for which there exists a polyhedral map on $S$
  with $p$-vector $p + c \cdot [(3k + 1) \times 3, 3k+5]$ and
  $v$-vector $v + d \cdot [4]$.
\end{theorem}
\begin{theorem}\label{eberhard:3:7:4}
  Let $p$ and $v$ be a pair of admissible sequences for an orientable
  closed $2$-manifold $S$ and $k \in \mathbb{N}$. Then there exists infinitely
  many $c, d \in \mathbb{N}$ for which there exist a polyhedral map on $S$
  with $p$-vector $p + c \cdot [(3k + 3) \times 3, 3k+7]$ and
  $v$-vector $v + d \cdot [4]$.
\end{theorem}
 
\section{Polyhedral maps and generalized Eberhard problems}\label{sec:polymap}
We will review basic notions from (topological) graph theory. A
\textbf{simple graph} $G$ is a finite undirected graph without loops and
multi-edges. If $G'$ is a subgraph of $G$ this is denoted by
$G' \subseteq G$. We want to write $u_1 - \dots - u_k$ for paths and
$u_1 - \dots - u_k - u_1$ for cycles. The \textbf{valence} of a vertex
is the number of incident edges.

All of our graphs are considered to be embedded into a closed
(topological) $2$-manifold, which we call \textbf{surfaces} for
brevity. We assume our $2$-manifolds to be oriented in this article. An
embedding of a simple graph with vertices $V$, edges $E$ and faces $F$
is called a \textbf{map}, provided that $G$ is simple, every vertex
$v \in V$ has valence at least $3$ and every $f \in F$ is a closed
$2$-cell (i.e. homeomorphic to a disk). The faces of a map incident to
$k$ edges (or equivalently, $k$ vertices) will be called
\textbf{$k$-gonal} faces or simply \textbf{$k$-gons}. A map on a closed
$2$-manifold is called \textbf{polyhedral}, if for every two faces
$f, f', f \neq f'$ there is either no vertex, a single vertex or a
single edge incident to both $f$ and $f'$. In these cases the two
faces are said to \textbf{meet properly}.

In Section \ref{sec:construction} we will weaken the definition of a
map to some extent to allow for $2$-valent vertices. This does not
warrant a whole new definition, so we state it here for completeness
and to avoid confusion.

An important property of polyhedral maps is that each edge contains
exactly two vertices and is contained in exactly two faces. From this
fact one can see that the concept of polyhedral maps dualizes
perfectly, i.e. if an embedding is a polyhedral map, then the dual of
the embedding is again a polyhedral map.

\begin{example}\label{rem:polymap:from:polyhedron} We can view every
  $3$-polytope as a map on a surface, where the graph of the map is
  the graph of the $3$-polytope and the embedding is held by radial
  projection onto $\mathbb{S}^2$. In this context each face of the
  $3$-polytope corresponds to one of the map. It is easy to see that
  this map is polyhedral, which gives rise to the property's
  name. Also note, that the dual map corresponding to a $3$-polytope
  is the corresponding map of the dual polytope.
\end{example}

To further strengthen the link between $3$-polytopes and polyhedral
maps on the sphere $\mathbb{S}^2$, we mention the following two results:

\begin{proposition} Every graph $G$ of a polyhedral map $M$
  is \textbf{$3$-connected}, i.e. $G$ has at least $3$ vertices and the
  deletion of any $2$ vertices leaves the graph connected.
\end{proposition}

\begin{theorem}[Steinitz's theorem] A
  graph is the edge graph of a $3$-polytope if and only if it is
  planar and $3$-connected.
\end{theorem}

These two theorems combined yield that every polyhedral map on
the sphere $\mathbb{S}^2$ comes from a $3$-polytope and vice versa.

We will now turn our focus to Eberhard theorems for polyhedral maps on
surfaces. The \textbf{$p$-vector} of a map $M$ on a $2$-manifold is the
sequence $(p_3, \dots, p_m)$, where $p_k$ denotes the number of faces
with exactly $k$ vertices. Similarly the \textbf{$v$-vector} of $M$ is
the sequence $(v_3, \dots, v_n)$ where each $v_k$ is the number of
vertices of $M$ with valence $k$.

A pair of sequences $(p, v)$ is said to be \textbf{realizable as a
  polyhedral map} on the closed $2$-manifold $S$ (or short: realizable
on $S$), if there exists such a map having $p$ as its $p$-vector and
$v$ as its $v$-vector. In this language we can state the following
two generalizations of Theorems~\ref{thm:eberhard:3} and
\ref{thm:eberhard:4}, which will be central in our constructions:

\begin{theorem}[Jendrol', Jucovi{\v{c}}~\cite{jendrol1977generalization}, 1977]
  \label{thm:eberhard:extended:3}
  Each pair of sequences $p = (p_3, \dots, p_m)$ and $v = (v_3, \dots,
  v_n)$ is realizable on a closed orientable $2$-manifold with Euler
  characteristic $\chi$ for some $p_6 \in \mathbb{N}$, $v_3 \in \mathbb{N}$ if and
  only if
  \begin{align*} \sum_{k=3}^m (6-k)p_k + 2 \sum_{k=4}^n (3-k)v_k =
    6\chi&, \\ \sum_{k=3 \atop 2 \nmid k}^{m} p_k \neq 0 \text{\quad or \quad}
    \sum_{k=4 \atop 3 \nmid k}^n v_k \neq 1& &\text{ if } \chi = 2, \\
    p \neq [5, 7] \text{\quad or \quad} v \neq [v_3 \times 3]& &\text{ if } \chi
    = 0.
  \end{align*}
\end{theorem}

\begin{theorem}[Barnette, Gr\"unbaum, Jendrol', Jucovi{\v{c}}, Zaks~\cite{
      jucovivc1973theorem, barnette1971toroidal, grunbaum1969planar,
      zaks1971analogue}, 1973] \label{thm:eberhard:extended:4} Each
  pair of sequences $p = (p_3, \dots, p_m)$ and $v = (v_3, \dots,
  v_n)$ is realizable on a closed orientable $2$-manifold with Euler
  characteristic $\chi$ for some $p_4, v_4 \in \mathbb{N}$ if and only if
  \begin{align*}
    \sum_{k=3}^m (4-k)p_k + \sum_{k = 3}^n (4-k)v_k = 4\chi&,\\
    \sum_{k=3}^m k p_k \equiv 0&& \pmod{2}, \\
    p \neq [3, 5] \text{\quad or \quad} v \neq [v_4 \times 4]& &\text{ if } \chi = 0,\\
    p \neq [p_4 \times 4] \text{\quad or \quad} v \neq [3, 5]& &\text{ if }
\chi = 0.
  \end{align*}
\end{theorem}

Note that special cases arise in both theorems if the surface is a
torus, i.e. if $\chi = 0$. Izmestiev et al.~\cite{izmestiev2013there}
gave a simple argument using holonomy groups to explain why these
special cases occur.

The rest of this section is devoted to find an easy characterization
for when we cannot hope Question~\ref{quest:gen:eberhard} to have a
positive answer. Let $M$ be a polyhedral map on a surface $S$ with
vertices $V$, edges $E$, and faces $F$, $p$-vector $p = (p_3, \dots, p_n)$
and $v$-vector $v = (v_3, \dots, v_m)$. Let $\chi = \chi(S)$ be the
Euler characteristic of $S$ and $e \defeq |E|$. The two basic
combinatorial results here are double-counting
\begin{align}
  2e = \sum_{k=3}^{m} k \cdot p_k = \sum_{k=3}^{n} k \cdot v_k, \label{eq:handshake}
\end{align}
and the Euler-Poincar\'e relation
\begin{align} 
 |V| - |E| + |F| = \chi. \label{eq:eulers:relation}
\end{align}

One can easily deduce from these relations, that the following two
equivalent conditions are necessary for two sequences $p$, $v$ being
the $p$- and $v$-vector of a polyhedral map:
\begin{proposition}
  Let $p$, $v$ be the $p$- and $v$-vector of a polyhedral map on a
  surface $S$. Then \eqref{eq:handshake} is true for some $e \in \mathbb{N}$
  and the following equivalent conditions hold:
  \begin{align*}
    \sum_{k=3}^m (6-k)p_k + 2 \sum_{k=4}^n (3-k)v_k &= 6\chi(S), \text{ and}\\
    \sum_{k=3}^m (4-k)p_k + \sum_{k = 3}^n (4-k)v_k &= 4\chi(S).
  \end{align*}
\end{proposition}
Equivalent here means, that together with
$\sum_{k \geq 3} p_k = \sum_{k \geq 3} v_k$ each equation can be
deduced from the other. If $p$ and $v$ satisfy these equations, we
will call the pair $(p, v)$ \textbf{admissible} (on $S$).  We remark
that we gain precisely the conditions of
Theorems~\ref{thm:eberhard:extended:3}
and~\ref{thm:eberhard:extended:4}.

In light of Question~\ref{quest:gen:eberhard} and using the same
arguments, it is not difficult to see, that we can always assume $p$
and $v$ to be admissible. Also important to note is, that from the
same arguments we can derive similar conditions on $q$ and $w$ which
have to be fulfilled in order for Question~\ref{quest:gen:eberhard} to
be answered in the positive. We will not go into the details here and
simply state the cases resulting from these restrictions.

We will restrict our setting to $q = [q_s \times s, q_l \times l]$
having only two non-negative entries and $w = [w_r \times r]$ having
one. Let us further assume that $\gcd(q_s, q_l) = 1$. These conditions
are quite natural, as any obstruction on finding a Eberhard-type
theorem for some $q$ will also give an obstruction for $c \cdot q$,
$c \in \mathbb{N}$. As stated above, not all values $s$, $l$, and $r$ can be
obtained in this setting, only the following cases can occur:
\begin{alignat*}{6}
  (s, r) = (3, 3): \quad && q = [q_3 \times 3, q_l \times l],\qquad&&w = [3],\qquad&q_3 = \tfrac{l - 6}{\gcd(l, 3)}, &\qquad&q_l = \tfrac{3}{\gcd(l, 3)} \\
  (s, r) = (4, 3): \quad && q = [q_4 \times 4, q_l \times l],\qquad&&w = [3],\qquad&q_4 = \tfrac{l - 6}{\gcd(l, 2)}, &\qquad&q_l = \tfrac{2}{\gcd(l, 2)} \\
  (s, r) = (5, 3): \quad && q = [q_5 \times 5, q_l \times l],\qquad&&w = [3],\qquad&q_5 = l - 6,                     &\qquad&q_l = 1 \\
  (s, r) = (3, 4): \quad && q = [q_3 \times 3, q_l \times l],\qquad&&w = [4],\qquad&q_3 = l - 4,                     &\qquad&q_l = 1 \\
  (s, r) = (3, 5): \quad && q = [q_3 \times 3, q_l \times l],\qquad&&w = [5],\qquad&q_3 = 3l - 10,                   &\qquad&q_l = 1
\end{alignat*}

We will answer the fourth case in this article. The full result
by the author is the following:

\begin{theorem}[Manecke~\cite{manecke_master_thesis}, 2016]
  Let $q = [q_s \times s, q_l \times l]$, $w = [r]$ as before. Then
  there exist inifitely many $c, d \in \mathbb{N}$ and a polyhedral map $M$ on a surface $S$
  for all admissible sequences $(p, v)$ if and only if $\gcd(s, l) = 1$
  and if $s = r = 3$, then $l < 11$.
\end{theorem}

We close the section by noting that by duality this result gives also
a full classification for $w$ having two non-zero entries and $q$
having only one.

\section{Construction}\label{sec:construction}

Let $r \in \mathbb{N}$ be the valence of those vertices we are free to add to
a polyhedral map. All constructions later in this article will utilize
the concept of replacing each face of a polyhedral map with a larger
patch. It can be quite challenging to see whether the resulting
structures fit together. This section introduces the necessary
formalism for these kinds of constructions. All statements are
presented without proof, all proofs can be found in
\cite{manecke_master_thesis}. Note that throughout this
section we allow $2$-valent vertices in special maps we call patches.

\begin{definition}[Patch] A map $\mathcal{P}$ on the Euclidean plane with
  possibly $2$-valent vertices on the unbounded \textbf{outer face} is
  called a \textbf{patch}. The vertices and edges of the outer face form
  the \textbf{boundary} $\partial\mathcal{P}$ of the patch. A patch is an
  \textbf{$r$-patch}, if each of its vertices except the ones on the
  boundary is $r$-valent, while for the valence $\deg(v)$ of a vertex
  $v$ on the outer face $2 \leq \deg(v) \leq r$ holds. The
  \textbf{$p$-vector of a patch} is the sequence $(p_3, p_4, \dots)$,
  where $p_k$ denotes the number of $k$-gonal inner faces of the
  patch.
\end{definition}

We say that two $r$-patches $\mathcal{P}_1$ and $\mathcal{P}_2$ fit
together along a path $v_1 - \dots - v_n$ on $\partial\mathcal{P}_1$
and $u_1 - \dots - u_n$ on $\partial\mathcal{P}_2$, if, after gluing
them together such that $v_i = u_{n+1-i}$ the resulting patch is still an
$r$-patch. Define $w(v) \defeq \deg(v) - 1$. Then the condition for fitting
together is just $w(v_i) + w(u_{n+1-i}) = r$ for all
$i \in \{2, \dots, n-1\}$ and $w(v_i) + w(u_{n+1-i}) \leq r$, $w(v_{n+1-i}) + w(u_i) \leq r$. We say a tuple $(w_1, \dots, w_n)$ is
\textbf{self-fitting}, if $w_i + w_{n+1-i} = r$ for all
$i \in \{1, \dots, n\}$.

Essential for our constructions will be $w$-expansions. We will use
them, when we replace all $k$-gons in a polyhedral map with larger structures.

\begin{definition}
  Let $w = (w_1, \dots, w_n) \in \mathbb{N}^n$. A
  \textbf{$w$-expansion} of an $r$-patch $\mathcal{P}$ with boundary
  $\partial\mathcal{P} = v_1 - v_2 - \dots - \dots - v_m - v_1$ is
  an $r$-patch $\mathcal{P'}$ with boundary
\begin{multline*}
  \partial\mathcal{P'} = {v'}_1 - \underparen{{v'}_1^{(1)} - \dots -
    {v'}_n^{(1)}} - {v'}_2 - \underparen{{v'}_1^{(2)} - \dots -
    {v'}_n^{(2)}} - \quad\dots \\ - {v'}_m - \underparen{{v'}_1^{(m)}
    - \dots - {v'}_n^{(m)}} - {v'}_1,
\end{multline*}
such that $w(v_i) = w(v'_i)$ and $w({v'}_j^{(i)}) = w_j$ for all
$i \in \{1, \dots, m\}$, $j \in \{1, \dots, n\}$. We call
the vertices $v'_i$ \textbf{corner vertices} and the vertices
${v'}_j^{(i)}$ \textbf{side vertices}. Furthermore, a patch is called
\textbf{$w$-$k$-gonal}, if it is the $w$-expansion of the patch
consisting of only a $k$-gon, i.e. if $w(v'_i) = 1$ for
$i \in \{1, \dots, k\}$.
\end{definition}

Using this notation we describe the following construction scheme:

\begin{algorithm}\label{alg:map}

  \noindent\textbf{Input:}
  A map on a surface $S$ with $p$-vector $p$, $v$-vector $v$,
  underlying graph $G = (V, E)$ and faces $F$.

  A self-fitting tuple $w = (w_1, \dots, w_n)$.

  For each $k$-gonal face $f \in F$ a $w$-$k$-gonal $r$-patch
  $\mathcal{P}(f)$ with $p$-vector $p^{(f)}$.

  \noindent\textbf{Output:}  
  A map on $S$ with $v$-vector $v + d \cdot [r]$ for some
  $d \in \mathbb{N}$ and $p$-vector $\sum_{f \in F} p^{(f)}$.

  \noindent\textbf{Description:}
  Divide each edge $e \in E$ in the embedding of $G$ in $S$ by $n$
  vertices and draw into each face $f$ the dedicated $r$-patch
  $\mathcal{P}(f)$ such that the corner vertices of $\mathcal{P}(f)$
  coincide with the original vertices $V$ and the side vertices are
  the new vertices added by the subdivision, see
  Fig.~\ref{fig:alg:map:overview}. Here we use the fact, that our
  surfaces are oriented and assume that all patches are glued with the
  same orientation. These patches form a combined graph, which is
  embedded by construction into $S$ (there is a homeomorphism between
  each subdivided face $f \in F$ and the corresponding patch
  $\mathcal{P}(f)$). It is straightforward to see, that this gives a
  map with the desired properties.
\end{algorithm}

\begin{figure}[!h]  
  \centering
  \begin{tikzpicture}[line cap=round, line join=round]
    \matrix (m) [column sep=1cm] {
      \begin{scope}[scale=0.4]
        \draw (-3,0)--(-3,3)--(-6,5)--(-6,8)--(-3,10)--(0,8)--(6,8)--(6,5)--(7,1)--(3,0)--(-3,0);
        \draw (-3,3)--(0,5)--(0,8);
        \draw (3,0)--(3,3)--(0,5);
        \draw (3,3)--(7,1);
        \draw (3,3)--(6,5);

        \draw (-3,0)-- (-3.5,-0.666);
        \draw (-3,3)-- (-3.5,2.5);
        \draw (-6,5)--  (-6.5,4.666);
        \draw (-6,8)-- (-6.5,8.333);
        \draw (-3,10)-- (-3,10.5);
        \draw (0,8)-- (0.5,8.333);
        \draw (6,8)-- (6.5,8.5);
        \draw (6,5)-- (6.5,5);
        \draw (7,1)-- (7.5,0.5);
        \draw (3,0)-- (3.5,-0.666);

        \fill[black] (-3,0) circle(3pt);
        \fill[black] (-3,3)circle(3pt);
        \fill[black] (-6,5)circle(3pt);
        \fill[black] (-6,8)circle(3pt);
        \fill[black] (-3,10)circle(3pt);
        \fill[black] (0,8)circle(3pt);
        \fill[black] (6,8)circle(3pt);
        \fill[black] (6,5)circle(3pt);
        \fill[black] (7,1) circle(3pt);
        \fill[black] (3,0)circle(3pt);
        \fill[black] (0,5)circle(3pt);
        \fill[black] (3,3)circle(3pt);

        \node at (-3,6.5) {$f_1$};
        \node at (0,2.5) {$f_2$};
        \node at (3,5.5) {$f_3$};
        \node at (4.5,1.5) {$f_4$};
        \node at (5,3) {$f_5$};
        
      \end{scope}
      &
      \begin{scope}[scale=0.4]
        \draw (-3,0)--(-3,3)--(-6,5)--(-6,8)--(-3,10)--(0,8)--(6,8)--(6,5)--(7,1)--(3,0)--(-3,0);
        \draw (-3,3)--(0,5)--(0,8);
        \draw (3,0)--(3,3)--(0,5);
        \draw (3,3)--(7,1);
        \draw (3,3)--(6,5);

        \draw (-3,0)-- (-3.5,-0.666);
        \draw (-3,3)-- (-3.5,2.5);
        \draw (-6,5)--  (-6.5,4.666);
        \draw (-6,8)-- (-6.5,8.333);
        \draw (-3,10)-- (-3,10.5);
        \draw (0,8)-- (0.5,8.333);
        \draw (6,8)-- (6.5,8.5);
        \draw (6,5)-- (6.5,5);
        \draw (7,1)-- (7.5,0.5);
        \draw (3,0)-- (3.5,-0.666);

        \fill[black] (-3,0) circle(3pt);
        \fill[black] (-3,3)circle(3pt);
        \fill[black] (-6,5)circle(3pt);
        \fill[black] (-6,8)circle(3pt);
        \fill[black] (-3,10)circle(3pt);
        \fill[black] (0,8)circle(3pt);
        \fill[black] (6,8)circle(3pt);
        \fill[black] (6,5)circle(3pt);
        \fill[black] (7,1) circle(3pt);
        \fill[black] (3,0)circle(3pt);
        \fill[black] (0,5)circle(3pt);
        \fill[black] (3,3)circle(3pt);

        \node at (-3,6.5) {$\mathcal{P}(f_1)$};
        \node at (0,2.5) {$\mathcal{P}(f_2)$};
        \node at (3,5.5) {$\mathcal{P}(f_3)$};
        \node at (4.3,1.2) {$\mathcal{P}(f_4)$};
        \node at (5,3) {$\mathcal{P}(f_5)$};
        
        \foreach \x in {0.2,0.4,0.6,0.8}
        \fill[black] ($(3,0)!\x!(-3,0)$) circle (3pt);    
        \foreach \x in {0.2,0.4,0.6,0.8}
        \fill[black] ($(-3,0)!\x!(-3,3)$) circle (3pt); 
        \foreach \x in {0.2,0.4,0.6,0.8}
        \fill[black] ($(-3,3)!\x!(-6,5)$) circle (3pt);
        \foreach \x in {0.2,0.4,0.6,0.8}
        \fill[black] ($(-6,5)!\x!(-6,8)$) circle (3pt);
        \foreach \x in {0.2,0.4,0.6,0.8}
        \fill[black] ($(-6,8)!\x!(-3,10)$) circle (3pt);
        \foreach \x in {0.2,0.4,0.6,0.8}
        \fill[black] ($(-3,10)!\x!(0,8)$) circle (3pt);
        \foreach \x in {0.2,0.4,0.6,0.8}
        \fill[black] ($(0,8)!\x!(6,8)$) circle (3pt);
        \foreach \x in {0.2,0.4,0.6,0.8}
        \fill[black] ($(6,8)!\x!(6,5)$) circle (3pt);
        \foreach \x in {0.2,0.4,0.6,0.8}
        \fill[black] ($(6,5)!\x!(7,1)$) circle (3pt);     
        \foreach \x in {0.2,0.4,0.6,0.8}
        \fill[black] ($(7,1)!\x!(3,0)$) circle (3pt); 
        \foreach \x in {0.2,0.4,0.6,0.8}
        \fill[black] ($(3,0)!\x!(3,3)$) circle (3pt);
        \foreach \x in {0.2,0.4,0.6,0.8}
        \fill[black] ($(3,3)!\x!(7,1)$) circle (3pt);
        \foreach \x in {0.2,0.4,0.6,0.8}
        \fill[black] ($(3,3)!\x!(6,5)$) circle (3pt);
        \foreach \x in {0.2,0.4,0.6,0.8}
        \fill[black] ($(3,3)!\x!(0,5)$) circle (3pt);
        \foreach \x in {0.2,0.4,0.6,0.8}
        \fill[black] ($(0,5)!\x!(0,8)$) circle (3pt);
        \foreach \x in {0.2,0.4,0.6,0.8}
        \fill[black] ($(0,5)!\x!(-3,3)$) circle (3pt);      
      \end{scope}
      \\
    };
  \end{tikzpicture}
  \caption{}
  \label{fig:alg:map:overview}
\end{figure}

As previously stated, these definitions are used to formalize the
construction step ``replace each face with a patch''. Up until now,
there is no requirement explicitly stated on the interior of the
patch. If we expect the result of such a construction to be a
polyhedral map, further conditions have to be met. Additionally, when
using Algorithm~\ref{alg:map} we have the problem of assigning a patch for
each face of the map. While we might need only one type of patch for a
$k$-gon for each $k \geq 3$, we could still have to deal with a huge
amount of values of $k$. We now want to define a construction scheme
for patches for arbitrary $k$ which additionally allow to create
polyhedral maps, even from non-polyhedral ones.

\begin{definition}\label{def:expansion:patch} Let $\mathcal{P}$ be an $r$-patch with boundary
\begin{multline*}
  \partial\mathcal{P} = i_0 - i_1 - \dots - i_{m-1} - (i_m = o_0) - \dots - o_s - \dots - o_{n - 1} - (o_n = i'_m) \\
  - i'_{m-1} - \dots - i'_1 - i'_0 - i_0
\end{multline*}
($i_m$ and $o_0$ denote the same vertex, the same holds for $o_n$ and
$i'_m$), $1 \leq s < n$, $m > 0$, such that:
  \begin{itemize}
  \item $w(i_0) + w(i'_0) = r - 1$,
  \item $\mathcal{P}$ fits to itself along $i_1 - \dots - i_{m-1}$ and
    $i'_{m-1} - \dots - i'_1$,
  \item $w(o_s) = 1$, and
  \item
    $(w(o_{s + 1}), \dots, w(o_{n - 1}), w(o_n) + w(o_0), w(o_1),
    \dots, w(o_{s - 1}))$ is a self-fitting tuple.
  \end{itemize}
  Such a patch will be called \textbf{expansion patch with outer tuple
    $(w(o_{s + 1}), \dots, \allowbreak w(o_{n - 1}), w(o_n) + w(o_0), w(o_1),
    \dots, w(o_{s - 1}))$}.
\end{definition}

\begin{example}\label{ex:easy:expansion1}
  We want to review the last definition with two examples. A hexagon
  can be interpreted as an expansion $3$-patch $\mathcal{H}$ with
  outer tuple
  $(w_{\mathcal{H}}(o_3) + w_{\mathcal{H}}(o_0), \allowbreak w_{\mathcal{H}}(o_1))
  = (2, 1)$, with vertices labeled according to
  Definition~\ref{def:expansion:patch} in
  Fig.~\ref{fig:easy:expansion1}. Similarly two quadrangles which share
  a common edge build an expansion $4$-patch $\mathcal{Q}_2$ with
  outer tuple
  $(w_{\mathcal{Q}_2}(o_3) + w_{\mathcal{Q}_2}(o_0),
  w_{\mathcal{Q}_2}(o_1)) = (2, 2)$, as seen in the same figure.

    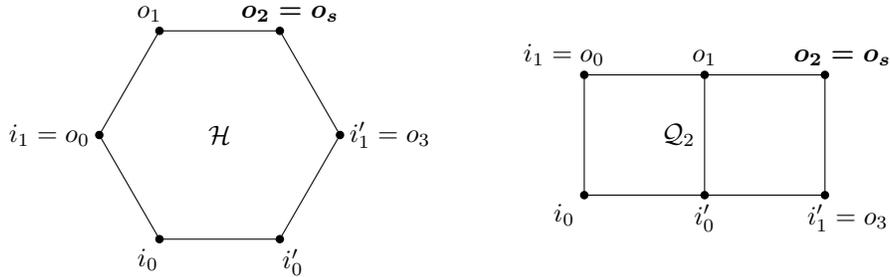
\begin{figure}[!h]  
      \centering
      \begin{tikzpicture}[line cap=round, line join=round]
    \matrix (m) [column sep=1cm] {
      \begin{scope}[scale=0.8]
        \draw (0 : 2) -- (60 : 2) -- (120 : 2) -- (180 : 2) -- (240 : 2) -- (300 : 2) -- (0 : 2);
        \fill [black] (  0 : 2) circle (2pt);
        \node[anchor=180] at (  0 : 2) {$i'_1 = o_3$};
        \fill [black] ( 60 : 2) circle (2pt);
        \node[anchor=240] at ( 60 : 2) {$\boldsymbol{o_2 = o_s}$};
        \fill [black] (120 : 2) circle (2pt);
        \node[anchor=300] at (120 : 2) {$o_1$};
        \fill [black] (180 : 2) circle (2pt);
        \node[anchor=  0] at (180 : 2) {$i_1 = o_0$};
        \fill [black] (240 : 2) circle (2pt);
        \node[anchor= 60] at (240 : 2) {$i_0$};
        \fill [black] (300 : 2) circle (2pt);
        \node[anchor=120] at (300 : 2) {$i'_0$};
        \node             at (  0 : 0) {$\mathcal{H}$};
      \end{scope}
      &
      \begin{scope}[scale=0.8]
        \draw (0, 1) -- (0, -1) -- (-2, -1) -- (-2, 1) -- (0, 1) -- (2, 1) -- (2, -1) -- (0, -1);
        \fill [black] (-2, -1) circle (2pt);
        \node[anchor= 45] at (-2, -1) {$i_0$};
        \fill [black] ( 0, -1) circle (2pt);
        \node[anchor= 90] at ( 0, -1) {$i'_0$};
        \fill [black] ( 2, -1) circle (2pt);
        \node[anchor=135] at ( 2, -1) {$i'_1 = o_3$};
        \fill [black] ( 2,  1) circle (2pt);
        \node[anchor=225] at ( 2,  1) {$\boldsymbol{o_2 = o_s}$};
        \fill [black] ( 0,  1) circle (2pt);
        \node[anchor=270] at ( 0,  1) {$o_1$};
        \fill [black] (-2,  1) circle (2pt);
        \node[anchor=315] at (-2,  1) {$i_1 = o_0$};
        \node[anchor=  0] at ( 0,  0) {$\mathcal{Q}_2$};
      \end{scope}
      \\
    };
  \end{tikzpicture}%
  \caption{Two expansion patches}
  \label{fig:easy:expansion1}
  \end{figure}
\end{example}

We want to pull apart this definition a bit to give a geometric
intuition. The first thing to note is that by definition we are able
to glue two copies of an expansion patch along the paths
$i_0 - \dots - i_{m}$ and $i'_{m} - \dots - i'_0$. When doing this the
new patch has a boundary path
$o_{s+1} - \dots - o_{n-1} - (o_n = o_0) - o_1 - \dots - o_{s-1}$,
which we require to be self-fitting. Therefore we can glue two of
those patches along this boundary to get an even larger patch (see
Fig.~\ref{fig:edge:patch}).
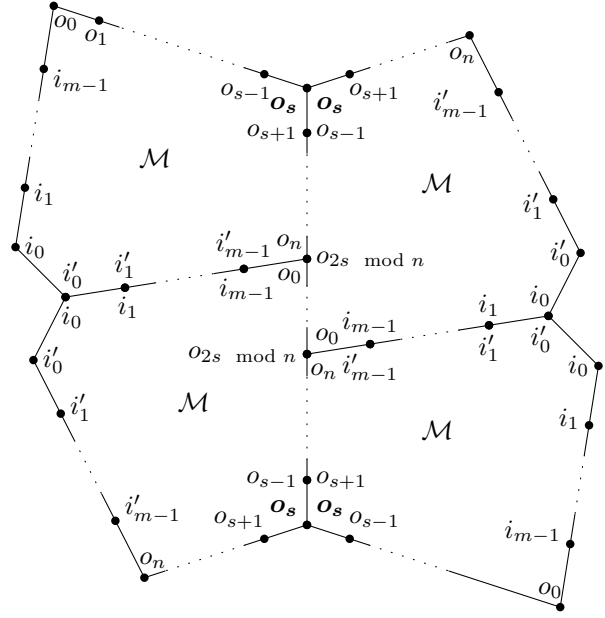
\begin{figure}[!h]  
  \centering
  \begin{tikzpicture}[line cap=round, line join=round]
    \begin{scope}[scale=0.8]
      \draw[shift={(-5,0)}] (9 : 1) -- (81 : 1)  (297 : 1) -- (9 : 1);

      \draw[shift={(-5,0)}] (9 : 1) -- (9 : 2.5);
      \draw[shift={(-5,0)}][loosely dotted] (9 : 2.5) -- (9 : 3.5);
      \draw[shift={(-5,0)}] (9 : 3.5) -- (9 : 5.062);
      \draw[shift={(-5,0)}] (81 : 1) -- (81 : 2.5);
      \draw[shift={(-5,0)}][loosely dotted] (81 : 2.5) -- (81 : 3.5);
      \draw[shift={(-5,0)}] (81 : 3.5) -- (81 : 5.062);
      \draw[shift={(-5,0)}] (297 : 1) -- (297 : 2.5);
      \draw[shift={(-5,0)}][loosely dotted] (297 : 2.5) -- (297 : 3.5);
      \draw[shift={(-5,0)}] (297 : 3.5) -- (297 : 5.062);
      \draw[shift={(-5,0)}][loosely dotted] (4 : 5.012) -- (356 : 5.012);

      \draw[shift={(-5,0)}] (-13 : 5.1315) -- (356 : 5.012) (4 : 5.012) -- (13 : 5.1315);
      \draw[shift={(-5,0)}][loosely dotted] (13 : 5.1315) -- (27 : 5.612);
      \draw[shift={(-5,0)}] (27 : 5.612) -- (36 : 6.180) -- (45 : 5.612);
      \draw[shift={(-5,0)}][loosely dotted] (45 : 5.612) -- (68 : 5.012);
      \draw[shift={(-5,0)}] (68 : 5.012) -- (81 : 5.062);
      \draw[shift={(-5,0)}] (-59 : 5.1315) -- (297 : 5.062);
      \draw[shift={(-5,0)}][loosely dotted] (-45 : 5.612) -- (-59 : 5.1315);
      \draw[shift={(-5,0)}] (-27 : 5.612) -- (-36 : 6.180) -- (-45 : 5.612);
      \draw[shift={(-5,0)}][loosely dotted] (-27 : 5.612) -- (-13 : 5.1315);

      \fill[shift={(-5,0)}] [black] (9 : 1) circle (2pt);
      \node[shift={(-4,0)}][anchor="108"] at (9 : 1) {$i_0$};
      \fill[shift={(-5,0)}] [black] (9 : 2) circle (2pt);
      \node[shift={(-4,0)}][anchor="99"] at (9 : 2) {$i_1$};
      \fill[shift={(-5,0)}] [black] (9 : 4) circle (2pt);
      \node[shift={(-4,0)}][anchor="99"] at (9 : 4) {$i_{m-1}$};
      \fill[shift={(-5,0)}] [black] (9 : 5.062) circle (2pt);
      \node[shift={(-4,0)}][anchor="45"] at (9 : 5.062) {$o_{0}$};
      \fill[shift={(-5,0)}] [black] (-9 : 5.062) circle (2pt);
      \node[shift={(-4,0)}][anchor="0"] at (-9 : 5.062) {$o_{2s \mod n}$};
      \fill[shift={(-5,0)}] [black] (-30 : 5.774) circle (2pt);
      \node[shift={(-4,0)}][anchor="0"] at (-30 : 5.774) {$o_{s - 1}$};
      \fill[shift={(-5,0)}] [black] (-36 : 6.180) circle (2pt);
      \node[shift={(-4,0)}][anchor="-36"] at (-36 : 6.180) {$\boldsymbol{o_s}$};
      \fill[shift={(-5,0)}] [black] (-42 : 5.774) circle (2pt);
      \node[shift={(-4,0)}][anchor="-36"] at (-42 : 5.774) {$o_{s + 1}$};
      \fill[shift={(-5,0)}] [black] (-63 : 5.062) circle (2pt);
      \node[shift={(-4,0)}][anchor="-117"] at (-63 : 5.062) {$o_{n}$};
      \fill[shift={(-5,0)}] [black] (-63 : 4) circle (2pt);
      \node[shift={(-4,0)}][anchor="198"] at (-63 : 4) {$i'_{m-1}$};
      \fill[shift={(-5,0)}] [black] (-63 : 2) circle (2pt);
      \node[shift={(-4,0)}][anchor="198"] at (-63 : 2) {$i'_{1}$};
      \fill[shift={(-5,0)}] [black] (-63 : 1) circle (2pt);
      \node[shift={(-4,0)}][anchor="180"] at (-63 : 1) {$i'_0$};
      \node[shift={(-4,0)}] at (-27 : 3.5) {$\mathcal{M}$};

      \fill[shift={(-5,0)}] [black] (81 : 1) circle (2pt);
      \node[shift={(-4,0)}][anchor="180"] at (81 : 1) {$i_0$};
      \fill[shift={(-5,0)}] [black] (81 : 2) circle (2pt);
      \node[shift={(-4,0)}][anchor="162"] at (81 : 2) {$i_1$};
      \fill[shift={(-5,0)}] [black] (81 : 4) circle (2pt);
      \node[shift={(-4,0)}][anchor="162"] at (81 : 4) {$i_{m-1}$};
      \fill[shift={(-5,0)}] [black] (81 : 5.062) circle (2pt);
      \node[shift={(-4,0)}][anchor="126"] at (81 : 5.062) {$o_{0}$};
      \fill[shift={(-5,0)}] [black] (72 : 5) circle (2pt);
      \node[shift={(-4,0)}][anchor="82"] at (72 : 5) {$o_{1}$};
      \fill[shift={(-5,0)}] [black] (42 : 5.774) circle (2pt);
      \node[shift={(-4,0)}][anchor="45"] at (42 : 5.774) {$o_{s - 1}$};
      \fill[shift={(-5,0)}] [black] (36 : 6.180) circle (2pt);
      \node[shift={(-4,0)}][anchor="36"] at (36 : 6.180) {$\boldsymbol{o_s}$};
      \fill[shift={(-5,0)}] [black] (30 : 5.774) circle (2pt);
      \node[shift={(-4,0)}][anchor="0"] at (30 : 5.774) {$o_{s + 1}$};
      \fill[shift={(-5,0)}] [black] (9 : 5.062) circle (2pt);
      \node[shift={(-4,0)}][anchor="-45"] at (9 : 5.062) {$o_{n}$};
      \node[shift={(-4,0)}][anchor="180"] at (9 : 5.062) {$o_{2s \mod n}$};
      
      \fill[shift={(-5,0)}] [black] (9 : 4) circle (2pt);
      \node[shift={(-4,0)}][anchor="270"] at (9 : 4) {$i'_{m-1}$};
      \fill[shift={(-5,0)}] [black] (9 : 2) circle (2pt);
      \node[shift={(-4,0)}][anchor="270"] at (9 : 2) {$i'_{1}$};
      \fill[shift={(-5,0)}] [black] (9 : 1) circle (2pt);
      \node[shift={(-4,0)}][anchor="252"] at (9 : 1) {$i'_0$};
      \node[shift={(-4,0)}] at (45 : 3.5) {$\mathcal{M}$};

      \draw[shift={(5,0)}, rotate around={180:(0,0)}] (9 : 1) -- (81 : 1) (297 : 1) -- (9 : 1);

      \draw[shift={(5,0)}, rotate around={180:(0,0)}] (9 : 1) -- (9 : 2.5);
      \draw[shift={(5,0)}, rotate around={180:(0,0)}][loosely dotted] (9 : 2.5) -- (9 : 3.5);
      \draw[shift={(5,0)}, rotate around={180:(0,0)}] (9 : 3.5) -- (9 : 5.062);
      \draw[shift={(5,0)}, rotate around={180:(0,0)}] (81 : 1) -- (81 : 2.5);
      \draw[shift={(5,0)}, rotate around={180:(0,0)}][loosely dotted] (81 : 2.5) -- (81 : 3.5);
      \draw[shift={(5,0)}, rotate around={180:(0,0)}] (81 : 3.5) -- (81 : 5.062);
      \draw[shift={(5,0)}, rotate around={180:(0,0)}] (297 : 1) -- (297 : 2.5);
      \draw[shift={(5,0)}, rotate around={180:(0,0)}][loosely dotted] (297 : 2.5) -- (297 : 3.5);
      \draw[shift={(5,0)}, rotate around={180:(0,0)}] (297 : 3.5) -- (297 : 5.062);
      
      \draw[shift={(5,0)}, rotate around={180:(0,0)}] (36 : 6.180) -- (45 : 5.612);
      \draw[shift={(5,0)}, rotate around={180:(0,0)}][loosely dotted] (45 : 5.612) -- (59 : 5.1315);
      \draw[shift={(5,0)}, rotate around={180:(0,0)}] (59 : 5.1315) -- (81 : 5.062);
      \draw[shift={(5,0)}, rotate around={180:(0,0)}] (-59 : 5.1315) -- (297 : 5.062);
      \draw[shift={(5,0)}, rotate around={180:(0,0)}][loosely dotted] (-45 : 5.612) -- (-59 : 5.1315);
      \draw[shift={(5,0)}, rotate around={180:(0,0)}] (-36 : 6.180) -- (-45 : 5.612);
      
      \fill[shift={(5,0)}, rotate around={180:(0,0)}] [black] (9 : 1) circle (2pt);
      \node[shift={(4,0)}][anchor="292"] at (189 : 1) {$i_0$};
      \fill[shift={(5,0)}, rotate around={180:(0,0)}] [black] (9 : 2) circle (2pt);
      \node[shift={(4,0)}][anchor="270"] at (189 : 2) {$i_1$};
      \fill[shift={(5,0)}, rotate around={180:(0,0)}] [black] (9 : 4) circle (2pt);
      \node[shift={(4,0)}][anchor="270"] at (189 : 4) {$i_{m-1}$};
      \fill[shift={(5,0)}, rotate around={180:(0,0)}] [black] (9 : 5.062) circle (2pt);
      \node[shift={(4,0)}][anchor="225"] at (189 : 5) {$o_{0}$};
      \fill[shift={(5,0)}, rotate around={180:(0,0)}] [black] (-30 : 5.774) circle (2pt);
      \node[shift={(4,0)}][anchor="180"] at (150 : 5.774) {$o_{s - 1}$};
      \fill[shift={(5,0)}, rotate around={180:(0,0)}] [black] (-36 : 6.180) circle (2pt);
      \node[shift={(4,0)}][anchor="144"] at (144 : 6.180) {$\boldsymbol{o_s}$};
      \fill[shift={(5,0)}, rotate around={180:(0,0)}] [black] (-42 : 5.774) circle (2pt);
      \node[shift={(4,0)}][anchor="144"] at (138 : 5.774) {$o_{s + 1}$};
      \fill[shift={(5,0)}, rotate around={180:(0,0)}] [black] (-63 : 5.062) circle (2pt);
      \node[shift={(4,0)}][anchor="63"] at (117 : 5) {$o_{n}$};
      \fill[shift={(5,0)}, rotate around={180:(0,0)}] [black] (-63 : 4) circle (2pt);
      \node[shift={(4,0)}][anchor="18"] at (117 : 4) {$i'_{m-1}$};
      \fill[shift={(5,0)}, rotate around={180:(0,0)}] [black] (-63 : 2) circle (2pt);
      \node[shift={(4,0)}][anchor="18"] at (117 : 2) {$i'_{1}$};
      \fill[shift={(5,0)}, rotate around={180:(0,0)}] [black] (-63 : 1) circle (2pt);
      \node[shift={(4,0)}][anchor="0"] at (117 : 1) {$i'_0$};
      \node[shift={(4,0)}] at (144 : 3.5) {$\mathcal{M}$};

      \fill[shift={(5,0)}, rotate around={180:(0,0)}] [black] (81 : 1) circle (2pt);
      \node[shift={(4,0)}][anchor="0"] at (261 : 1) {$i_0$};
      \fill[shift={(5,0)}, rotate around={180:(0,0)}] [black] (81 : 2) circle (2pt);
      \node[shift={(4,0)}][anchor="342"] at (261 : 2) {$i_1$};
      \fill[shift={(5,0)}, rotate around={180:(0,0)}] [black] (81 : 4) circle (2pt);
      \node[shift={(4,0)}][anchor="342"] at (261 : 4) {$i_{m-1}$};
      \fill[shift={(5,0)}, rotate around={180:(0,0)}] [black] (81 : 5.062) circle (2pt);
      \node[shift={(4,0)}][anchor="300"] at (261 : 5.062) {$o_{0}$};
      \fill[shift={(5,0)}, rotate around={180:(0,0)}] [black] (42 : 5.774) circle (2pt);
      \node[shift={(4,0)}][anchor="216"] at (222 : 5.774) {$o_{s - 1}$};
      \fill[shift={(5,0)}, rotate around={180:(0,0)}] [black] (36 : 6.180) circle (2pt);
      \node[shift={(4,0)}][anchor="216"] at (216 : 6.180) {$\boldsymbol{o_s}$};
      \fill[shift={(5,0)}, rotate around={180:(0,0)}] [black] (30 : 5.774) circle (2pt);
      \node[shift={(4,0)}][anchor="180"] at (210 : 5.774) {$o_{s + 1}$};
      \fill[shift={(5,0)}, rotate around={180:(0,0)}] [black] (9 : 5.062) circle (2pt);
      \node[shift={(4,0)}][anchor="135"] at (189 : 5.062) {$o_{n}$};
      \fill[shift={(5,0)}, rotate around={180:(0,0)}] [black] (9 : 4) circle (2pt);
      \node[shift={(4,0)}][anchor="90"] at (189 : 4) {$i'_{m-1}$};
      \fill[shift={(5,0)}, rotate around={180:(0,0)}] [black] (9 : 2) circle (2pt);
      \node[shift={(4,0)}][anchor="90"] at (189 : 2) {$i'_{1}$};
      \fill[shift={(5,0)}, rotate around={180:(0,0)}] [black] (9 : 1) circle (2pt);
      \node[shift={(4,0)}][anchor="72"] at (189 : 1) {$i'_0$};
      \node[shift={(4,0)}] at (216 : 3.5) {$\mathcal{M}$};
    \end{scope}
  \end{tikzpicture}
  \caption{An edge patch}
  \label{fig:edge:patch}
\end{figure}

For an expansion patch $\mathcal{M}$, we want to call the patch obtained
by gluing four copies of $\mathcal{M}$ as stated the \textbf{edge patch}
of $\mathcal{M}$. An expansion patch will be said to have the
\textbf{polyhedral property} if every two inner faces in the
corresponding edge patch meet properly.

\begin{example}\label{ex:easy:expansion2}
  The examples in Example~\ref{ex:easy:expansion1} do in fact have the
  polyhedral property, which can be verified by looking at the edge
  patch in Fig.~\ref{fig:easy:expansion2}.
  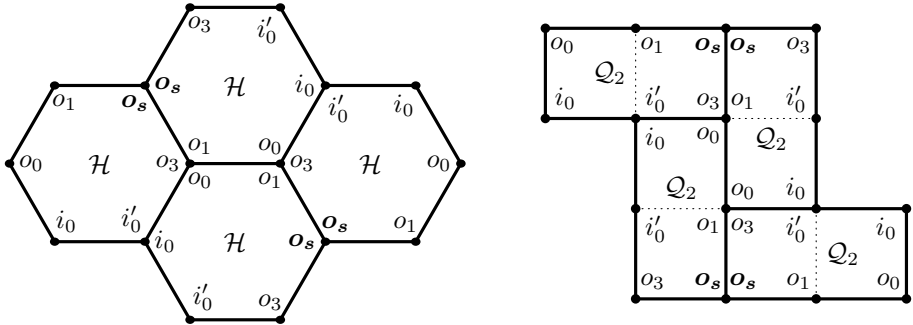
\begin{figure}[!h]  
    \centering
    \begin{tikzpicture}[line cap=round, line join=round]
  
    \matrix (m) [column sep=1cm] {
      \begin{scope}[xscale=1.0, yscale=0.866, scale=0.6]
        \draw[very thick] (-1,  0) -- (-2, -2) -- (-4, -2) -- (-5,  0) -- (-4,  2) -- (-2,  2) -- (-1,  0);
        \draw[very thick] (-1,  0) -- ( 1,  0) -- ( 2, -2) -- ( 1, -4) -- (-1, -4) -- (-2, -2);
        \draw[very thick] ( 1,  0) -- ( 2,  2) -- ( 1,  4) -- (-1,  4) -- (-2,  2);
        \draw[very thick] ( 2, -2) -- ( 4, -2) -- ( 5,  0) -- ( 4,  2) -- ( 2,  2);

        \fill[black] (-4,-2) circle(3pt);
        \fill[black] ( 4, 2) circle(3pt);
        \fill[black] (-4, 2) circle(3pt);
        \fill[black] (-2,-2) circle(3pt);
        \fill[black] (-2, 2) circle(3pt);
        \fill[black] ( 2,-2) circle(3pt);
        \fill[black] ( 5, 0) circle(3pt);
        \fill[black] (-5, 0) circle(3pt);
        \fill[black] (-1, 0) circle(3pt);
        \fill[black] ( 1, 0) circle(3pt);
        \fill[black] (-1,-4) circle(3pt);
        \fill[black] ( 1,-4) circle(3pt);
        \fill[black] (-1, 4) circle(3pt);
        \fill[black] ( 2, 2) circle(3pt);
        \fill[black] ( 4,-2) circle(3pt);
        \fill[black] ( 1, 4) circle(3pt);

        \node[anchor=240] at (-4, -2) {$i_0$};
        \node[anchor=180] at (-5,  0) {$o_0$};
        \node[anchor=120] at (-4,  2) {$o_1$};
        \node[anchor= 60] at (-2,  2) {$\boldsymbol{o_s}$};
        \node[anchor=  0] at (-1,  0) {$o_3$};
        \node[anchor=300] at (-2, -2) {$i'_0$};

        \node[anchor=240] at (-1, -4) {$i'_0$};
        \node[anchor=180] at (-2, -2) {$i_0$};
        \node[anchor=120] at (-1,  0) {$o_0$};
        \node[anchor= 60] at ( 1,  0) {$o_1$};
        \node[anchor=  0] at ( 2, -2) {$\boldsymbol{o_s}$};
        \node[anchor=300] at ( 1, -4) {$o_3$};

        \node[anchor=240] at (-1,  0) {$o_1$};
        \node[anchor=180] at (-2,  2) {$\boldsymbol{o_s}$};
        \node[anchor=120] at (-1,  4) {$o_3$};
        \node[anchor= 60] at ( 1,  4) {$i'_0$};
        \node[anchor=  0] at ( 2,  2) {$i_0$};
        \node[anchor=300] at ( 1,  0) {$o_0$};

        \node[anchor=240] at ( 2, -2) {$\boldsymbol{o_s}$};
        \node[anchor=180] at ( 1,  0) {$o_3$};
        \node[anchor=120] at ( 2,  2) {$i'_0$};
        \node[anchor= 60] at ( 4,  2) {$i_0$};
        \node[anchor=  0] at ( 5,  0) {$o_0$};
        \node[anchor=300] at ( 4, -2) {$o_1$};

        \node             at (-3,  0) {$\mathcal{H}$};
        \node             at ( 0, -2) {$\mathcal{H}$};
        \node             at ( 0,  2) {$\mathcal{H}$};
        \node             at ( 3,  0) {$\mathcal{H}$};
      \end{scope}
      &
      \begin{scope}[scale=0.6]
        \draw[very thick] (-2, -3) -- (-2,  1) -- (-4,  1) -- (-4,  3) -- (2,  3) -- (2, -1) -- (4, -1) -- (4, -3) -- (-2, -3);
        \draw[very thick] ( 0, -3) -- ( 0,  3);
        \draw[very thick] ( 2, -1) -- ( 0, -1);
        \draw[very thick] ( 0,  1) -- (-2,  1);
        \draw[dotted] (-2, -1) -- ( 0, -1);
        \draw[dotted] ( 2,  1) -- ( 0,  1);
        \draw[dotted] (-2,  3) -- (-2,  1);
        \draw[dotted] ( 2, -3) -- ( 2, -1);

        \foreach \x in {-2,0,2}
        \foreach \y in {-3,-1,1,3}
        \fill[black] (\x, \y) circle(3pt);
        \fill[black] (4,-1) circle(3pt);
        \fill[black] (4,-3) circle(3pt);
        \fill[black] (-4, 1) circle(3pt);
        \fill[black] (-4, 3) circle(3pt);

        \node[anchor=225] at (-4,  1) {$i_0$};
        \node[anchor=135] at (-4,  3) {$o_0$};
        \node[anchor=135] at (-2,  3) {$o_1$};
        \node[anchor= 45] at ( 0,  3) {$\boldsymbol{o_s}$};
        \node[anchor=315] at ( 0,  1) {$o_3$};
        \node[anchor=225] at (-2,  1) {$i'_0$};
        \node[anchor=  0] at (-2,  2) {$\mathcal{Q}_2$};

        \node[anchor= 45] at ( 4, -1) {$i_0$};
        \node[anchor=315] at ( 4, -3) {$o_0$};
        \node[anchor=315] at ( 2, -3) {$o_1$};
        \node[anchor=225] at ( 0, -3) {$\boldsymbol{o_s}$};
        \node[anchor=135] at ( 0, -1) {$o_3$};
        \node[anchor= 45] at ( 2, -1) {$i'_0$};
        \node[anchor=180] at ( 2, -2) {$\mathcal{Q}_2$};

        \node[anchor=135] at (-2,  1) {$i_0$};
        \node[anchor= 45] at ( 0,  1) {$o_0$};
        \node[anchor= 45] at ( 0, -1) {$o_1$};
        \node[anchor=315] at ( 0, -3) {$\boldsymbol{o_s}$};
        \node[anchor=225] at (-2, -3) {$o_3$};
        \node[anchor=135] at (-2, -1) {$i'_0$};
        \node[anchor=270] at (-1, -1) {$\mathcal{Q}_2$};

        \node[anchor=315] at ( 2, -1) {$i_0$};
        \node[anchor=225] at ( 0, -1) {$o_0$};
        \node[anchor=225] at ( 0,  1) {$o_1$};
        \node[anchor=135] at ( 0,  3) {$\boldsymbol{o_s}$};
        \node[anchor= 45] at ( 2,  3) {$o_3$};
        \node[anchor=315] at ( 2,  1) {$i'_0$};
        \node[anchor= 90] at ( 1,  1) {$\mathcal{Q}_2$};
      \end{scope}
      \\
    };
  \end{tikzpicture}%
  \caption{Two edge patches}
  \label{fig:easy:expansion2}
  \end{figure}
\end{example}

Expansion patches will be our basic building block for all our constructive proofs. We can use them to obtain larger $o$-$k$-gonal patches for any $k \geq 3$:

\begin{algorithm}\label{alg:expansion:patch}
  
  \noindent\textbf{Input:}
  An expansion $r$-patch $\mathcal{M}$ with outer tuple $o$ and $p$-vector $p$.

  \noindent\textbf{Output:}
  For every $k \geq 3$ an $o$-$k$-gonal $r$-patch $\mathcal{M}(k)$ with $p$-vector $[k] + k \cdot p$.

  If $\mathcal{M}$ has the polyhedral property, then all inner faces of $\mathcal{M}(k)$ meet properly.

  \noindent\textbf{Description:}
  We construct $\mathcal{M}(k)$ from $k$ copies of $\mathcal{M}$ and a
  single $k$-gon. Let
\begin{multline*}
  \partial\mathcal{P} = i_0 - i_1 - \dots - (i_m = o_0) - o_1 - \dots - o_s - \dots - o_{n - 1} - (o_n = i'_m) - \dots \\ - i'_1 - i'_0 - i_0
\end{multline*}
be the boundary of $\mathcal{M}$ as in
Definition~\ref{def:expansion:patch}. We now form a larger patch by
gluing the edge $\{i_0, i'_0\}$ of each of the $k$ copies of
$\mathcal{M}$ to an edge of the $k$-gon and also gluing the vertex
associated to $i_l$, $1 \leq l \leq m$ from one copy to the vertex
associated to $i'_l$ from the adjacent copy. Graphically speaking, we
form a ring of $k$ patches of the form $\mathcal{M}$ around the
$k$-gon. The $p$-vector of $\mathcal{M}(k)$ is therefore
$[k] + k \cdot p$. We leave out the proof that the inner faces of
$\mathcal{M}(k)$ meet properly in the case of $\mathcal{M}$ being
polyhedral.

\end{algorithm}

With these constructions at hand we can now finally design a scheme to
create a polyhedral map from a non-polyhedral one.

\begin{proposition}\label{thm:alg:polymap}
  Given a map $M$ on an orientable closed $2$-manifold $S$ and an
  expansion $r$-patch $\mathcal{M}$ with outer tuple $o$,
  Algorithm~\ref{alg:map} returns a polyhedral map on $S$ when we take
  $\mathcal{P}(f) = \mathcal{M}(k)$ for each $k$-gonal face $f$ of
  $M$.
\end{proposition}

\begin{figure}[!h]
  \centering
  \begin{tikzpicture}[line cap=round, line join=round, scale=0.8]
    \matrix (m) [column sep=1cm] {
      \begin{scope}[scale=0.4]
        \draw (-3,0)--(-3,3)--(-6,5)--(-6,8)--(-3,10)--(0,8)--(6,8)--(6,5)--(7,1)--(3,0)--(-3,0);
        \draw (-3,3)--(0,5)--(0,8);
        \draw (3,0)--(3,3)--(0,5);
        \draw (3,3)--(7,1);
        \draw (3,3)--(6,5);

        \draw (-3,0)-- (-3.5,-0.666);
        \draw (-3,3)-- (-3.5,2.5);
        \draw (-6,5)--  (-6.5,4.666);
        \draw (-6,8)-- (-6.5,8.333);
        \draw (-3,10)-- (-3,10.5);
        \draw (0,8)-- (0.5,8.333);
        \draw (6,8)-- (6.5,8.5);
        \draw (6,5)-- (6.5,5);
        \draw (7,1)-- (7.5,0.5);
        \draw (3,0)-- (3.5,-0.666);

        \fill[black] (-3,0) circle(3pt);
        \fill[black] (-3,3)circle(3pt);
        \fill[black] (-6,5)circle(3pt);
        \fill[black] (-6,8)circle(3pt);
        \fill[black] (-3,10)circle(3pt);
        \fill[black] (0,8)circle(3pt);
        \fill[black] (6,8)circle(3pt);
        \fill[black] (6,5)circle(3pt);
        \fill[black] (7,1) circle(3pt);
        \fill[black] (3,0)circle(3pt);
        \fill[black] (0,5)circle(3pt);
        \fill[black] (3,3)circle(3pt);

        \node at (-3,6.5) {$f_1$};
        \node at (0,2.5) {$f_2$};
        \node at (3,5.5) {$f_3$};
        \node at (4.5,1.5) {$f_4$};
        \node at (5,3) {$f_5$};
        
      \end{scope}
      &
      \begin{scope}[scale=0.4]
        \draw (-3,0)--(-3,3)--(-6,5)--(-6,8)--(-3,10)--(0,8)--(6,8)--(6,5)--(7,1)--(3,0)--(-3,0);
        \draw (-3,3)--(0,5)--(0,8);
        \draw (3,0)--(3,3)--(0,5);
        \draw (3,3)--(7,1);
        \draw (3,3)--(6,5);

        \coordinate[shift={(0,0.8)}] (penta1) at ( 50:0.5) ;
        \coordinate[shift={(0,0.8)}] (penta2) at (122:0.5) ;
        \coordinate[shift={(0,0.8)}] (penta3) at (194:0.5) ;
        \coordinate[shift={(0,0.8)}] (penta4) at (266:0.5) ;
        \coordinate[shift={(0,0.8)}] (penta5) at (338:0.5) ;

        \coordinate[shift={(1.2,2.4)}] (penta9) at  ( 20:0.5) ;
        \coordinate[shift={(1.2,2.4)}] (penta10) at ( 92:0.5) ;
        \coordinate[shift={(1.2,2.4)}] (penta11) at (164:0.5) ;
        \coordinate[shift={(1.2,2.4)}] (penta12) at (236:0.5) ;
        \coordinate[shift={(1.2,2.4)}] (penta13) at (308:0.5) ;

        \coordinate[shift={(-1.2,2.6)}] (hexa1) at (  0:0.5) ;
        \coordinate[shift={(-1.2,2.6)}] (hexa2) at ( 60:0.5) ;
        \coordinate[shift={(-1.2,2.6)}] (hexa3) at (120:0.5) ;
        \coordinate[shift={(-1.2,2.6)}] (hexa4) at (180:0.5) ;
        \coordinate[shift={(-1.2,2.6)}] (hexa5) at (240:0.5) ;
        \coordinate[shift={(-1.2,2.6)}] (hexa6) at (300:0.5) ;

        \coordinate[shift={(1.76,0.52)}] (tri1) at ( 75:0.4) ;
        \coordinate[shift={(1.76,0.52)}] (tri2) at (195:0.4) ;
        \coordinate[shift={(1.76,0.52)}] (tri3) at (315:0.4) ;
        
        \coordinate[shift={(2.16,1.2)}] (tri4) at ( 30:0.4) ;
        \coordinate[shift={(2.16,1.2)}] (tri5) at (150:0.4) ;
        \coordinate[shift={(2.16,1.2)}] (tri6) at (270:0.4) ;

        \draw (penta5)--(penta1)--(penta2)--(penta3)--(penta4)--(penta5);
        \draw (penta9)--(penta10)--(penta11)--(penta12)--(penta13)--(penta9);

        \draw (hexa1)--(hexa2)--(hexa3)--(hexa4)--(hexa5)--(hexa6)--(hexa1);

        \draw (tri1)--(tri2)--(tri3)--(tri1);
        \draw (tri4)--(tri5)--(tri6)--(tri4);
        
        \draw (-3,0)-- (-3.5,-0.666);
        \draw (-3,3)-- (-3.5,2.5);
        \draw (-6,5)--  (-6.5,4.666);
        \draw (-6,8)-- (-6.5,8.333);
        \draw (-3,10)-- (-3,10.5);
        \draw (0,8)-- (0.5,8.333);
        \draw (6,8)-- (6.5,8.5);
        \draw (6,5)-- (6.5,5);
        \draw (7,1)-- (7.5,0.5);
        \draw (3,0)-- (3.5,-0.666);

        \fill[black] (-3,0) circle(3pt);
        \fill[black] (-3,3)circle(3pt);
        \fill[black] (-6,5)circle(3pt);
        \fill[black] (-6,8)circle(3pt);
        \fill[black] (-3,10)circle(3pt);
        \fill[black] (0,8)circle(3pt);
        \fill[black] (6,8)circle(3pt);
        \fill[black] (6,5)circle(3pt);
        \fill[black] (7,1) circle(3pt);
        \fill[black] (3,0)circle(3pt);
        \fill[black] (0,5)circle(3pt);
        \fill[black] (3,3)circle(3pt);

        \draw ($(3,0)!0.4!(3,3)$) --(tri2);
        \draw ($(7,1)!0.6!(3,3)$) --(tri1);
        \draw ($(7,1)!0.4!(3,0)$) --(tri3);

        \draw ($(6,5)!0.6!(3,3)$) --(tri5);
        \draw ($(7,1)!0.4!(3,3)$) --(tri6);
        \draw ($(6,5)!0.4!(7,1)$) --(tri4);

        \draw ($(3,3)!0.6!(0,5)$) --(penta1);
        \draw ($(0,5)!0.6!(-3,3)$) --(penta2);
        \draw ($(-3,3)!0.6!(-3,0)$) --(penta3);
        \draw ($(-3,0)!0.6!(3,0)$) --(penta4);
        \draw ($(3,0)!0.6!(3,3)$) --(penta5);

        \draw ($(6,5)!0.6!(6,8)$) --(penta9);
        \draw ($(6,8)!0.6!(0,8)$) --(penta10);
        \draw ($(0,8)!0.6!(0,5)$) --(penta11);
        \draw ($(0,5)!0.6!(3,3)$) --(penta12);
        \draw ($(3,3)!0.6!(6,5)$) --(penta13);

        \draw ($(0,5)!0.6!(0,8)$) --(hexa1);
        \draw ($(0,8)!0.6!(-3,10)$) --(hexa2);
        \draw ($(-3,10)!0.6!(-6,8)$) --(hexa3);
        \draw ($(-6,8)!0.6!(-6,5)$) --(hexa4);
        \draw ($(-6,5)!0.6!(-3,3)$) --(hexa5);
        \draw ($(-3,3)!0.6!(0,5)$) --(hexa6);

        \node at (-1.5,1) {$\mathcal{M}$};
        \node at (-1.7,2.4) {$\mathcal{M}$};
        \node at (1.5,0.8) {$\mathcal{M}$};
        \node at (1.6,2.6) {$\mathcal{M}$};
        \node at (-0.2,3.8) {$\mathcal{M}$};

        \node at (3.8,0.7) {$\mathcal{M}$};
        \node at (3.8,2) {$\mathcal{M}$};
        \node at (5.45,1.2) {$\mathcal{M}$};

        \node at (6,2.5) {$\mathcal{M}$};
        \node at (4.4,3) {$\mathcal{M}$};
        \node at (5.6,3.8) {$\mathcal{M}$};

        \node at (3,4.5) {$\mathcal{M}$};
        \node at (4.8,5.8) {$\mathcal{M}$};
        \node at (4.3,7.3) {$\mathcal{M}$};
        \node at (1.4,7.3) {$\mathcal{M}$};
        \node at (1.4,5.5) {$\mathcal{M}$};

        \node[shift={(-1.2,2.6)}] at ( 30:1.8) {$\mathcal{M}$};
        \node[shift={(-1.2,2.6)}] at ( 90:1.8) {$\mathcal{M}$} ;
        \node[shift={(-1.2,2.6)}] at (150:1.8) {$\mathcal{M}$} ;
        \node[shift={(-1.2,2.6)}] at (210:1.8) {$\mathcal{M}$} ;
        \node[shift={(-1.2,2.6)}] at (270:1.8) {$\mathcal{M}$} ;
        \node[shift={(-1.2,2.6)}] at (330:1.8) {$\mathcal{M}$} ;

        \foreach \x in {0.2,0.4,0.6,0.8}
        \fill[black] ($(3,0)!\x!(-3,0)$) circle (3pt);    
        \foreach \x in {0.2,0.4,0.6,0.8}
        \fill[black] ($(-3,0)!\x!(-3,3)$) circle (3pt); 
        \foreach \x in {0.2,0.4,0.6,0.8}
        \fill[black] ($(-3,3)!\x!(-6,5)$) circle (3pt);
        \foreach \x in {0.2,0.4,0.6,0.8}
        \fill[black] ($(-6,5)!\x!(-6,8)$) circle (3pt);
        \foreach \x in {0.2,0.4,0.6,0.8}
        \fill[black] ($(-6,8)!\x!(-3,10)$) circle (3pt);
        \foreach \x in {0.2,0.4,0.6,0.8}
        \fill[black] ($(-3,10)!\x!(0,8)$) circle (3pt);
        \foreach \x in {0.2,0.4,0.6,0.8}
        \fill[black] ($(0,8)!\x!(6,8)$) circle (3pt);
        \foreach \x in {0.2,0.4,0.6,0.8}
        \fill[black] ($(6,8)!\x!(6,5)$) circle (3pt);
        \foreach \x in {0.2,0.4,0.6,0.8}
        \fill[black] ($(6,5)!\x!(7,1)$) circle (3pt);     
        \foreach \x in {0.2,0.4,0.6,0.8}
        \fill[black] ($(7,1)!\x!(3,0)$) circle (3pt); 
        \foreach \x in {0.2,0.4,0.6,0.8}
        \fill[black] ($(3,0)!\x!(3,3)$) circle (3pt);
        \foreach \x in {0.2,0.4,0.6,0.8}
        \fill[black] ($(3,3)!\x!(7,1)$) circle (3pt);
        \foreach \x in {0.2,0.4,0.6,0.8}
        \fill[black] ($(3,3)!\x!(6,5)$) circle (3pt);
        \foreach \x in {0.2,0.4,0.6,0.8}
        \fill[black] ($(3,3)!\x!(0,5)$) circle (3pt);
        \foreach \x in {0.2,0.4,0.6,0.8}
        \fill[black] ($(0,5)!\x!(0,8)$) circle (3pt);
        \foreach \x in {0.2,0.4,0.6,0.8}
        \fill[black] ($(0,5)!\x!(-3,3)$) circle (3pt);      
      \end{scope}
      \\
    };
  \end{tikzpicture}
  \caption{Algorithm~\ref{alg:map} for
    $\mathcal{P}(f) = \mathcal{M}(k)$}
  \label{fig:alg:polymap:overview}
\end{figure}
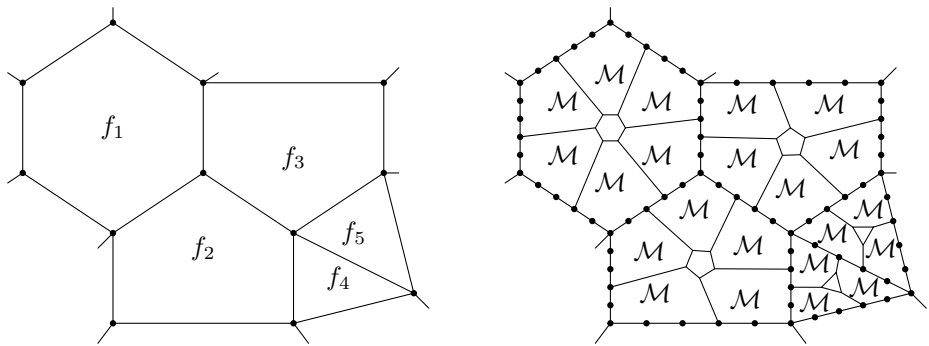

\begin{example}\label{ex:easy:expansion3}
  Using the expansion patches $\mathcal{H}$ and $\mathcal{Q}_2$ from
  Examples~\ref{ex:easy:expansion1} and \ref{ex:easy:expansion2}, we
  can construct from a polyhedral map a new one with arbitrarily many
  hexagons (or quadrangles) added, while inserting only $3$-valent (or
  $4$-valent) vertices. Given a map $M$ on a closed oriented
  $2$-manifold, we can simply use Theorem~\ref{thm:alg:polymap}
  repeatedly on $M$ with either $\mathcal{H}$ or $\mathcal{Q}_2$ to
  get the desired result. The theorem inserts at least a single
  hexagon or quadrangle during each step (which is quite an
  understatement, the number of polygons added is by far larger), so
  repeating this step eventually leads to a map that has more than
  a specified amount of hexagons or quadrangles.
\end{example}

Putting all these constructions together, we can formulate a proof
strategy for Question~\ref{quest:gen:eberhard} in the case of
$q = [q_s \times s, q_l \times l]$, $w = [r]$. We state it for $r = 4$
only, but the ideas carry over to $r = 3$, too.
\begin{proposition}\label{thm:main:const}
  Let $r = 4$, $p = (p_3, \dots, p_m)$ and $v = (v_3, \dots v_n)$ be
  an admissible pair of sequences. Let $w = [4]$ and
  $q = [q_s \times s, q_l \times l]$, where $s = 3$, $l > 4$,
  $q_s = l - 4$ and $q_l = 1$. Assume there exist
  \begin{itemize}
  \item an expansion $r$-patch $\mathcal{P}_N$ with outer tuple $o$
    consisting of $s$-gons and $l$-gons,
  \item an $o$-$4$-gonal $r$-patch $\mathcal{P}_F$ consisting of
    $s$-gons and $l$-gons, and
  \item an expansion $r$-patch $\mathcal{P}_P$ with the polyhedral
    property consisting of $s$-gons and $l$-gons.
  \end{itemize}
  Then there exists a polyhedral map on $S$ with $p$-vector
  $p + c \cdot q$ and $v$-vector $v + d \cdot w$ for infinitely many
  $c, d \in \mathbb{N}$.
\end{proposition}
\begin{proof}[Idea of the proof]
  Use Theorem~\ref{thm:eberhard:extended:3} or
  Theorem~\ref{thm:eberhard:extended:4} as a starting map and apply
  Proposition~\ref{thm:main:const} for the given patches.
\end{proof}

\begin{remark}
  We will use Theorem~\ref{thm:main:const} heavily in the next
  section. Therefore, we want to stress what is needed to check to see
  if the prerequisites of Theorem~\ref{thm:main:const} are
  fulfilled. As stated, we need three $r$-patches, $\mathcal{P}_N$,
  $\mathcal{P}_F$ and $\mathcal{P}_P$, which are called in this
  manner for the rest of the article. The list of properties is:
  \begin{itemize}
  \item $\mathcal{P}_N$, $\mathcal{P}_F$ and $\mathcal{P}_P$ consist
    of only $s$-gons and $l$-gons and all inner vertices have valence
    $r$.
    \item $\mathcal{P}_N$ and $\mathcal{P}_P$ are expansion patches:
      \begin{itemize}
      \item $i_0$ and $i'_0$ are in sum incident to $r-1$ faces,
      \item $i_k$ and $i'_k$, are in sum incident to $r$ faces,
        $1 \leq k < m$,
      \item starting at the vertex $o_s$ and going in both directions
        for each pair of vertices $o_{s+k \mod n}$ and
        $o_{s-k \mod n}$ the identity
        $w_{\mathcal{P}_X}(o_{s+k \mod n}) + w_{\mathcal{P}_X}(o_{s-k
          \mod n}) = r$ holds, where we ``identify'' $o_n$ with
        $o_0$. For ease of comparison we also state the outer tuple
        $o$ for $\mathcal{P}_N$.
      \end{itemize}
    \item $\mathcal{P}_F$ is $o$-$4$-gonal, i.e. if
      starting at some vertex and looking at the number of inner faces
      incident to this vertex we see the pattern
      \begin{align*}
        1, o_1, \dots o_n, 1, o_1, \dots o_n, 1, o_1, \dots o_n, 1, o_1, \dots o_n, 
      \end{align*}
      where $o$ is the outer tuple of $\mathcal{P}_N$.
    \item $\mathcal{P}_P$ has the polyhedral property. For this we
      provide the corresponding edge patch to make the verification
      easier.
 \end{itemize}
\end{remark}

\section{$4$-valent Eberhard-type theorems with triangles}\label{sec:3:4}

In this section we want to prove $4$-valent Eberhard-type
theorems with triangles, i.e. for $q = [q_3 \times 3, q_l \times l]$,
$w = [w_4 \times 4]$, $l > 4$, $\gcd(q_3, q_l) = 1$. 

For all the proofs, we want to use Theorem \ref{thm:main:const},
therefore we need to have a construction scheme for patches with
arbitrarily large $l$-gons. These we get by the next three
constructions:
\begin{algorithm}\label{alg:edge:replacement:3:4:1} When we want
  to use this construction in this section we label an edge (the
  specified edge) with a square and point with arrows to a $k_1$-gon
  and a $k_2$-gon.

  \noindent\textbf{Input:}
  A $4$-patch with $p$-vector $p$ and a specified edge with
  exactly one vertex incident to some $k_1$-gon and the other vertex
  incident to some $k_2$-gon. We require that in the cyclic order
  around both end points, starting at the specified edge, the
  $k_1$-gon and the $k_2$-gon have the same position.

  \noindent\textbf{Output:}
  A new $4$-patch with $p$-vector $p - [k_1, k_2] + [2k \times 3] + [k_1 + k, k_2 + k]$ for all $k \in \mathbb{N}$.

  If every two faces of the $4$-patch meet properly, then this is carried over to the new patch.

  \noindent\textbf{Description:}
  Using the replacement of the single edge as seen in
  Fig. \ref{fig:alg:edge:replacement:3:4:1} results in a new
  $4$-patch with $p$-vector
  $p - [k_1, k_2] + [2 \times 3] + [k_1 + 1, k_2 + 1]$. The line on
  the left labeled with a square is the specified edge and the line
  on the right labeled with a square is a new edge that we can use
  to repeat the construction. Every time we use this construction we
  add two new triangles while increasing the number of vertices of
  the left and right polygon by one; doing this $k$ times gives the
  desired $4$-patch. That all faces meet properly follows by
  induction as this property is preserved in each step.
  \begin{figure}[!h]
    \centering
    \begin{tikzpicture}[line cap=round, line join=round]
    \matrix (m) [column sep=1cm] {
      \begin{scope}[scale=0.8]
        \draw[lsquare] (-1, 0) -- (1, 0);
        \draw (-1.2, 0.5) -- (-1, 0) -- (-1.2, -0.5);
        \draw (1.2, 0.5) -- (1, 0) -- (1.2, -0.5);
        \draw (-1, 0) -- (-0.8, -0.5);
        \draw (1, 0) -- (0.8, 0.5);
        \node (k1) at (-2, 0) {$k_1$};
        \node (k2) at (2, 0) {$k_2$};
        \draw[lface] (-1, 0) -- (k1);
        \draw[lface] ( 1, 0) -- (k2);
        \fill[black] (-1,0) circle(2pt);
        \fill[black] (1,0) circle(2pt);
        
      \end{scope}
      &
      \begin{scope}[scale=0.8]
        \draw[lsquare] (-1, 0.5) -- (1, 0.5);
        \draw (-1, -0.5) -- (1, -0.5);
        \draw (-1.2, 1) -- (-1, 0.5) -- (-1, -0.5) -- (-1.2, -1);
        \draw (1.2, 1) -- (1, 0.5) -- (1, -0.5) -- (1.2, -1);
        \draw (-1, 0.5) -- (1, -0.5);
        \draw (-1, -0.5) -- (-0.8, -1);
        \draw (1, 0.5) -- (0.8, 1);
        \node (k1) at (-2, 0) {$k_1$};
        \node (k2) at (2, 0) {$k_2$};
        \draw[lface] (-1, 0.5) -- (k1);
        \draw[lface] ( 1, 0.5) -- (k2);

        \fill[black] (-1,0.5) circle(2pt);
        \fill[black] (1,0.5) circle(2pt);
        
        \fill[black] (-1,-0.5) circle(2pt);
        \fill[black] (1,-0.5) circle(2pt);

        \node at (-0.7,-0.1) {$3$};
        \node at (0.7,0.1) {$3$};

      \end{scope}
      \\
    };
  \end{tikzpicture}
  \caption{}
  \label{fig:alg:edge:replacement:3:4:1}
\end{figure}
\end{algorithm}

\begin{algorithm}\label{alg:edge:replacement:3:4:2} When we want to
  use this construction in this section we label an edge (the
  specified edge) with a diamond and point with arrows to a $k_1$-gon
  and a $k_2$-gon.
  
  \noindent\textbf{Input:}
  A $4$-patch with $p$-vector $p$ and a specified edge which is the
  common edge of a $k_1$-gon and a $k_2$-gon.

  \noindent\textbf{Output:}
  A $4$-patch with $p$-vector
  $p - [k_1, k_2] + [(6k) \times 3] + [k_1 + 3k, k_2 + 3k]$ for all
  $k \in \mathbb{N}$.

  \noindent\textbf{Description:}
  Using the replacement of the single edge as seen in
  Fig. \ref{fig:alg:edge:replacement:3:4:2} results in a new
  $4$-patch with $p$-vector
  $p - [k_1, k_2] + [(6k) \times 3] + [k_1 + 3, k_2 + 3]$. The line on
  the left labeled with a diamond is the specified edge and the line
  on the right labeled with a diamond is a new edge which we can use
  to repeat the construction. Every time we use this construction we
  add six triangles while increasing the number of vertices of the
  left and right polygon by three; doing this $k$ times gives the
  desired $4$-patch.
  \begin{figure}[!h]
    \centering
    \begin{tikzpicture}[line cap=round, line join=round]
    \matrix (m) [column sep=1cm] {
      \begin{scope}[scale=1.2]
        \draw[ldiamond] (0, -1) -- (0, 1);
        \draw (0.5, -1.2) -- (0, -1) -- (-0.5, -1.2);
        \draw (0.5, 1.2) -- (0, 1) -- (-0.5, 1.2);
        \node (k1) at (-1, 0) {$k_1$};
        \node (k2) at (1, 0) {$k_2$};
        \draw[lface] (0, 0) -- (k1);
        \draw[lface] (0, 0) -- (k2);

        \fill[black] (0,-1) circle(1.5pt);
        \fill[black] (0,1) circle(1.5pt);
        
      \end{scope}
      &
      \begin{scope}[scale=1.2]
        \draw[ldiamond] (0, 0.6) -- (0, 1.2);
        \draw (0, 0.6) -- (0, -1);
        \draw (0.5, -1.2) -- (0, -1) -- (-0.5, -1.2);
        \draw (0.5, 1.4) -- (0, 1.2) -- (-0.5, 1.4);
        
        \draw (-0.5, 0) -- (0, -0.6) -- (0.5, 0);
        \draw (-0.5, 0) -- (0, -0.2) -- (0.5, 0);
        \draw (-0.5, 0) -- (0, 0.6) -- (0.5, 0);
        \draw (-0.5, 0) -- (0, 0.2) -- (0.5, 0);

        \node (k1) at (-2, 0) {$k_1 + 3$};
        \node (k2) at (2, 0) {$k_2 + 3$};
        \draw[lface] (0, 0.9) -- (k1);
        \draw[lface] (0, 0.9) -- (k2);

        \fill[black] (0,-1) circle(1pt);
        \fill[black] (0,1.2) circle(1pt);
        \fill[black] (0,0.2) circle(1pt);
        \fill[black] (0,0.6) circle(1pt);
        \fill[black] (0,-0.6) circle(1pt);
        \fill[black] (0,-0.2) circle(1pt);
        \fill[black] (0.5,0) circle(1pt);
        \fill[black] (-0.5,0) circle(1pt);

        \node at (-0.1,0) {$3$};
        \node at (0.1,0) {$3$};
        
        \node at (-0.12,0.3) {$3$};
        \node at (0.12,0.3) {$3$};
        \node at (-0.12,-0.3) {$3$};
        \node at (0.12,-0.3) {$3$};

      \end{scope}
      \\
    };
  \end{tikzpicture}
  \caption{}
  \label{fig:alg:edge:replacement:3:4:2}
  \end{figure}
\end{algorithm}

\begin{algorithm}\label{alg:edge:replacement:3:4:3}
  When we want to use this construction in this section we encircle a
  vertex (the specified vertex) and point with arrows to a $k_1$-gon
  and a $k_2$-gon.

  \noindent \textbf{Input:} A $4$-patch with $p$-vector $p$ and a
  specified vertex which is adjacent to both $k_1$-gon and a $k_2$-gon
  which do not share an edge containing this vertex.

  \noindent \textbf{Output:}
  A new $4$-patch with $p$-vector $p - [k_1, k_2] + [(6k) \times 3] +
  [k_1 + 3k, k_2 + 3k]$ for all $k \in \mathbb{N}$.

  \noindent \textbf{Description:}
  Using the replacement of the vertex as seen in Figure
  \ref{fig:alg:edge:replacement:3:4:3} results in a new $4$-patch with
  $p$-vector $p - [k_1, k_2] + [(6k) \times 3] + [k_1 + 3, k_2 +
    3]$. The encircled vertex on the left is the specified vertex and
  the encircled vertex on the right is a new vertex which we can use
  to repeat the construction. Every time we use this construction we
  add six triangles while increasing the number of vertices of the
  left and right polygon by three; doing this $k$ times gives the
  desired $4$-patch.
\end{algorithm}
\begin{figure}[!h]
  \centering
  \begin{tikzpicture}[line cap=round, line join=round]
    \matrix (m) [column sep=1cm] {
      \begin{scope}
        \draw (0.2, -0.5) -- (0, 0) -- (-0.2, -0.5);
        \draw (0.2, 0.5) -- (0, 0) -- (-0.2, 0.5);
        \node (k1) at (-1, 0) {$k_1$};
        \node (k2) at (1, 0) {$k_2$};
        \node[lvertex] at (0, 0) {};
        \draw[lface] (0, 0) -- (k1);
        \draw[lface] (0, 0) -- (k2);

        \fill[black] (0,0) circle(1.5pt);
        
      \end{scope}
      &
      \begin{scope}[scale=1.2]
        \draw (0.2, -1.5) -- (0, -1) -- (-0.2, -1.5);
        \draw (0.2, 1.5) -- (0, 1) -- (-0.2, 1.5);
        \draw (0, -1) -- (-0.25, -0.5) -- (-0.25, 0.5) -- (0, 1);
        \draw (0, -1) -- ( 0.25, -0.5) -- ( 0.25, 0.5) -- (0, 1);
        \draw (-0.25, -0.5) -- ( 0.25,  0.5) -- (-0.25,  0.5);
        \draw (-0.25,  0.5) -- ( 0.25, -0.5) -- (-0.25, -0.5);
        \node (k1) at (-2, 0) {$k_1 + 3$};
        \node (k2) at (2, 0) {$k_2 + 3$};
        \node[lvertex] at (0, 1) {};
        \draw[lface] (0, 1) -- (k1);
        \draw[lface] (0, 1) -- (k2);

        \fill[black] (0,-1) circle(1pt);
        \fill[black] (0,1) circle(1pt);
        \fill[black] (0,0) circle(1pt);
        \fill[black] (0.25,0.5) circle(1pt);
        \fill[black] (0.25,-0.5) circle(1pt);
        \fill[black] (-0.25,0.5) circle(1pt);
        \fill[black] (-0.25,-0.5) circle(1pt);
        
        \node at (-0.15,0) {$3$};
        \node at (0.15,0) {$3$};

        \node at (0,0.3) {$3$};
        \node at (0,-0.3) {$3$};
        \node at (0,0.7) {$3$};
        \node at (0,-0.7) {$3$};
      \end{scope}
      \\
    };
  \end{tikzpicture}
  \caption{}
  \label{fig:alg:edge:replacement:3:4:3}
\end{figure}

\begin{figure}[b!]
  \begin{center}
    \parbox{0.45\textwidth}{
      \center
      \begin{tikzpicture}[line cap=round, line join=round]
        \begin{scope}[scale=0.7, yscale=0.866]
          \draw (0.5,1) -- (-0.5,1) -- (-1.5,3) -- (-4,6) -- (-3.5,7)
          -- (-2.5,7) -- (-1,6.5) -- (-0.5,7) -- (0.5,7) -- (1,7.5) --
          (2.5,7) -- (1.5,3) -- (0.5,1); \draw
          (-1.5,3)--(-0.5,4.5)--(0.5,3)--(0.5,1); %
          \draw (-4,6)--(-1,6)--(-1,6.5); %
          \draw[lsquare] (-1.5,3)--(-1,6); %
          \draw (-0.5,4.5) --(-1,6); %
          \draw (-0.5,4.5) --(0.5,7); %
          \draw[lsquare] (0.5,3) --(0.5,7); %
          \draw (0.5,3) --(1,7.5); %

          \fill[black] (0.5,1)    circle (2pt);
          \fill[black] (-0.5,1)   circle (2pt);
          \fill[black] (-1.5,3)   circle (2pt);
          \fill[black] (-4,6)     circle (2pt);
          \fill[black] (-3.5,7)   circle (2pt);
          \fill[black] (-2.5,7)   circle (2pt);
          \fill[black] (-1,6.5)   circle (2pt);
          \fill[black] (-0.5,7)   circle (2pt);
          \fill[black] (0.5,7)    circle (2pt);
          \fill[black] (1,7.5)    circle (2pt);
          \fill[black] (2.5,7)    circle (2pt);
          \fill[black] (1.5,3)    circle (2pt);
          \fill[black] (-0.5,4.5) circle (2pt);
          \fill[black] (0.5,3)    circle (2pt);
          \fill[black] (-1,6)     circle (2pt);
          
          \node (n1) at (-2  ,6.5) {$5$};
          \node (n2) at (-0.5,2.5) {$5$};
          \node (n3) at ( 1,4) {$5$};
          \node (n4) at (-0.25,6.25) {$5$};
          \node at (0.75,6.75) {$3$};
          \node at (-1.8,4.5) {$3$};
          \node at (-0.75,4.5) {$3$};
          \node at (0,4.5) {$3$};

          \draw[lface] (-1,6)--(n1);
          \draw[lface] (-1.5,3)--(n2);
          \draw[lface] (0.5,3)--(n3);
          \draw[lface] (0.5,7)--(n4);
          
          \node[anchor= 90] at (0.5,1)  {$i_{0}'$};
          \node[anchor= 90] at (-0.5,1) {$i_0$};
          \node[anchor=  0] at (-1.5,3) {$i_1$};
          \node[anchor= 30] at (-4,6)   {$i_2=o_0$};
          \node[anchor=300] at (-3.5,7) {$\boldsymbol{o_s}$};
          \node[anchor=270] at (-2.5,7) {$o_2$};
          \node[anchor=315] at (-1,6.5) {$o_3$};
          \node[anchor=270] at (-0.5,7) {$o_4$};
          \node[anchor=270] at (0.5,7)  {$o_5$};
          \node[anchor=270] at (1,7.5)  {$o_6$};
          \node[anchor=240] at (2.5,7)  {$i_2'=o_7$};
          \node[anchor=180] at (1.5,3)  {$i_1'$};
        \end{scope}
      \end{tikzpicture}
      {\par\noindent\centering\footnotesize (a) $\mathcal{P}_N=\mathcal{P}_P$}
    }
    \hspace*{30pt}
    \parbox{0.45\textwidth}{
      \center
      \begin{tikzpicture}[line cap=round, line join=round]
        \begin{scope}[scale=0.4]
          \begin{scope}[yscale=0.866]
            \draw[very thick] (0.5,1)--(-0.5,1)--(-1.5,3)--(-4,6)--(-3.5,7)--(-2.5,7)--(-1,6.5)--(-0.5,7)--(0.5,7)--(1,7.5)--(2.5,7)--(1.5,3)--(0.5,1);
            \draw (-1.5,3)--(-0.5,4.5)--(0.5,3)--(0.5,1); %
            \draw (-4,6)--(-1,6)--(-1,6.5); %
            \draw (-1.5,3)--(-1,6); %
            \draw (-0.5,4.5) --(-1,6); %
            \draw (-0.5,4.5) --(0.5,7); %
            \draw (0.5,3) --(0.5,7); %
            \draw (0.5,3) --(1,7.5); %

            \fill[black] (0.5,1)    circle (3pt);
            \fill[black] (-0.5,1)   circle (3pt);
            \fill[black] (-1.5,3)   circle (3pt);
            \fill[black] (-4,6)     circle (3pt);
            \fill[black] (-3.5,7)   circle (3pt);
            \fill[black] (-2.5,7)   circle (3pt);
            \fill[black] (-1,6.5)   circle (3pt);
            \fill[black] (-0.5,7)   circle (3pt);
            \fill[black] (0.5,7)    circle (3pt);
            \fill[black] (1,7.5)    circle (3pt);
            \fill[black] (2.5,7)    circle (3pt);
            \fill[black] (1.5,3)    circle (3pt);
            \fill[black] (-0.5,4.5) circle (3pt);
            \fill[black] (0.5,3)    circle (3pt);
            \fill[black] (-1,6)     circle (3pt);
          \end{scope}
          \begin{scope}[rotate=-60, yscale=0.866]
            \draw[very thick] (0.5,1)--(-0.5,1)--(-1.5,3)--(-4,6)--(-3.5,7)--(-2.5,7)--(-1,6.5)--(-0.5,7)--(0.5,7)--(1,7.5)--(2.5,7)--(1.5,3)--(0.5,1);
            \draw (-1.5,3)--(-0.5,4.5)--(0.5,3)--(0.5,1); %
            \draw (-4,6)--(-1,6)--(-1,6.5); %
            \draw (-1.5,3)--(-1,6); %
            \draw (-0.5,4.5) --(-1,6); %
            \draw (-0.5,4.5) --(0.5,7); %
            \draw (0.5,3) --(0.5,7); %
            \draw (0.5,3) --(1,7.5); %

            \fill[black] (0.5,1)    circle (3pt);
            \fill[black] (-0.5,1)   circle (3pt);
            \fill[black] (-1.5,3)   circle (3pt);
            \fill[black] (-4,6)     circle (3pt);
            \fill[black] (-3.5,7)   circle (3pt);
            \fill[black] (-2.5,7)   circle (3pt);
            \fill[black] (-1,6.5)   circle (3pt);
            \fill[black] (-0.5,7)   circle (3pt);
            \fill[black] (0.5,7)    circle (3pt);
            \fill[black] (1,7.5)    circle (3pt);
            \fill[black] (2.5,7)    circle (3pt);
            \fill[black] (1.5,3)    circle (3pt);
            \fill[black] (-0.5,4.5) circle (3pt);
            \fill[black] (0.5,3)    circle (3pt);
            \fill[black] (-1,6)     circle (3pt);
          \end{scope}
          \begin{scope}[yscale=0.866,shift={(0 cm,14 cm)},rotate=180]
            \draw[very thick] (0.5,1)--(-0.5,1)--(-1.5,3)--(-4,6)--(-3.5,7)--(-2.5,7)--(-1,6.5)--(-0.5,7)--(0.5,7)--(1,7.5)--(2.5,7)--(1.5,3)--(0.5,1);
            \draw (-1.5,3)--(-0.5,4.5)--(0.5,3)--(0.5,1); %
            \draw (-4,6)--(-1,6)--(-1,6.5); %
            \draw (-1.5,3)--(-1,6); %
            \draw (-0.5,4.5) --(-1,6); %
            \draw (-0.5,4.5) --(0.5,7); %
            \draw (0.5,3) --(0.5,7); %
            \draw (0.5,3) --(1,7.5); %

            \fill[black] (0.5,1)    circle (3pt);
            \fill[black] (-0.5,1)   circle (3pt);
            \fill[black] (-1.5,3)   circle (3pt);
            \fill[black] (-4,6)     circle (3pt);
            \fill[black] (-3.5,7)   circle (3pt);
            \fill[black] (-2.5,7)   circle (3pt);
            \fill[black] (-1,6.5)   circle (3pt);
            \fill[black] (-0.5,7)   circle (3pt);
            \fill[black] (0.5,7)    circle (3pt);
            \fill[black] (1,7.5)    circle (3pt);
            \fill[black] (2.5,7)    circle (3pt);
            \fill[black] (1.5,3)    circle (3pt);
            \fill[black] (-0.5,4.5) circle (3pt);
            \fill[black] (0.5,3)    circle (3pt);
            \fill[black] (-1,6)     circle (3pt);
          \end{scope}
          \begin{scope}[shift={(0 cm,12.124 cm)},rotate=120,yscale=0.866]
            \draw[very thick] (0.5,1)--(-0.5,1)--(-1.5,3)--(-4,6)--(-3.5,7)--(-2.5,7)--(-1,6.5)--(-0.5,7)--(0.5,7)--(1,7.5)--(2.5,7)--(1.5,3)--(0.5,1);
            \draw (-1.5,3)--(-0.5,4.5)--(0.5,3)--(0.5,1); %
            \draw (-4,6)--(-1,6)--(-1,6.5); %
            \draw (-1.5,3)--(-1,6); %
            \draw (-0.5,4.5) --(-1,6); %
            \draw (-0.5,4.5) --(0.5,7); %
            \draw (0.5,3) --(0.5,7); %
            \draw (0.5,3) --(1,7.5); %
            \draw (0.5,1)--(-0.5,1)--(-1.5,3)--(-4,6)--(-3.5,7)--(-2.5,7)--(-1,6.5)--(-0.5,7)--(0.5,7)--(1,7.5)--(2.5,7)--(1.5,3)--(0.5,1);

            \fill[black] (0.5,1)    circle (3pt);
            \fill[black] (-0.5,1)   circle (3pt);
            \fill[black] (-1.5,3)   circle (3pt);
            \fill[black] (-4,6)     circle (3pt);
            \fill[black] (-3.5,7)   circle (3pt);
            \fill[black] (-2.5,7)   circle (3pt);
            \fill[black] (-1,6.5)   circle (3pt);
            \fill[black] (-0.5,7)   circle (3pt);
            \fill[black] (0.5,7)    circle (3pt);
            \fill[black] (1,7.5)    circle (3pt);
            \fill[black] (2.5,7)    circle (3pt);
            \fill[black] (1.5,3)    circle (3pt);
            \fill[black] (-0.5,4.5) circle (3pt);
            \fill[black] (0.5,3)    circle (3pt);
            \fill[black] (-1,6)     circle (3pt);
          \end{scope}
        \end{scope}
      \end{tikzpicture}
      {\par\noindent\centering\footnotesize (b) Edge patch of $\mathcal{P}_N$}
    }

    \vspace*{15px}
    \parbox{1.0\textwidth}{
      \center
      \begin{tikzpicture}[line cap=round, line join=round]
        \begin{scope}[yscale=0.866]
          \draw(0,1.5) -- (-1,2) -- (-0.5,3) -- (0.5,3.5) -- (1.5,3) --
          (2.5,3) -- (4.5,2.5) -- (5.5,3) -- (6.5,2.5) -- (5.5,2) --
          (5,1) -- (4,0.5) -- (3,0.5) -- (2,0) -- (2,1) -- (1,1.5) --
          (0,1.5); \draw (3,0.5)--(2,1)--(2,2)--(1,1.5);
          \draw (1.5,3) -- (2,2) -- (2.5,3);
          \draw (5,1) -- (4.5,2.5) -- (5.5,2) -- (5.5,3);
          
          \node (n1) at (0.5,2.5) {$7$};
          \node (n2) at (3.5,1.5) {$7$};
          \node at (2.25,0.5) {$3$};
          \node at (1.75,1.5) {$3$};
          \node at (2,2.666) {$3$};
          \node at (5,1.8) {$3$};
          \node at (5.2,2.5) {$3$};
          \node at (5.8,2.5) {$3$};

          \draw[lface](2,2)--(n1);
          \draw[lface](2,2)--(n2);
          
          \fill[black] (0,1.5)   circle (2pt);
          \fill[black] (-1,2)    circle (2pt);
          \fill[black] (-0.5,3)  circle (2pt);  
          \fill[black] (0.5,3.5) circle (2pt); 
          \fill[black] (1.5,3)   circle (2pt); 
          \fill[black] (2.5,3)   circle (2pt); 
          \fill[black] (4.5,2.5) circle (2pt); 
          \fill[black] (5.5,3)   circle (2pt); 
          \fill[black] (6.5,2.5) circle (2pt); 
          \fill[black] (5.5,2)   circle (2pt);
          \fill[black] (5,1)     circle (2pt);
          \fill[black] (4,0.5)   circle (2pt);
          \fill[black] (3,0.5)   circle (2pt);
          \fill[black] (2,0)     circle (2pt);
          \fill[black] (2,1)     circle (2pt);
          \fill[black] (1,1.5)   circle (2pt); 
          \fill[black] (2,2)     circle (2pt);
          
          \node[lvertex] at (2,2) {};
          \node[anchor= 90] at (0,1.5)   {$i_{0}$};
          \node[anchor=  0] at (-1,2)    {$i_{1}$};
          \node[anchor=330] at (-0.5,3)  {$i_{2}=o_{0}$};
          \node[anchor=270] at (0.5,3.5) {$\boldsymbol{o_{s}}$};
          \node[anchor=240] at (1.5,3)   {$o_{2}$};
          \node[anchor=270] at (2.5,3)   {$o_{3}$};
          \node[anchor=270] at (4.5,2.5) {$o_{4}$};
          \node[anchor=270] at (5.5,3)   {$o_{5}$};
          \node[anchor=180] at (6.5,2.5) {$o_{6}$};
          \node[anchor=150] at (5.5,2)   {$o_{7}$};
          \node[anchor=140] at (5,1)     {$o_{8}$};
          \node[anchor=150] at (4,0.5)   {$o_{9}$};
          \node[anchor= 90] at (3,0.5)   {$o_{10}$};
          \node[anchor= 90] at (2,0)     {$i_{2}'=o_{11}$};
          \node[anchor= 45] at (2,1)     {$i_{1}'$};
          \node[anchor= 90] at (1,1.5)   {$i_{0}'$};

        \end{scope}
      \end{tikzpicture}
      {\par\noindent\centering\footnotesize (c) $\mathcal{P}_N$}
    }
    \vspace*{15px}
    
    \parbox{0.45\textwidth}{
      \center
      \begin{tikzpicture}[line cap=round, line join=round]
        \begin{scope}[scale=3]
          \node (n1) at (0.065666, -0.104494) {3};
          \node (n2) at (-0.505302, -0.226440) {5};
          \node (n3) at (-0.168506, -0.785841) {3};
          \node (n4) at (0.257283, -0.572559) {5};
          \node (n5) at (0.781133, -0.189136) {3};
          \node (n6) at (0.515595, 0.210065) {5};
          \node (n7) at (-0.328177, 0.522290) {5};
          \node (n8) at (-0.124628, 0.912568) {3};
          \node (n9) at (-0.628342, 0.681971) {5};
          \node[anchor=280] (n10) at (-0.324759, 0.928107) {3};
          \node (n11) at (-0.581971, -0.728342) {5};
          \node[anchor= 10] (n12) at (-0.928107, -0.324759) {3};
          \node (n13) at (0.748342, -0.571971) {5};
          \node[anchor=100] (n14) at (0.324759, -0.928107) {3};
          \node (n15) at (0.581971, 0.728342) {5};
          \node[anchor=190] (n16) at (0.928107, 0.324759) {3};

          \fill[black] (-0.145531, 0.231610) circle (0.2pt);
          \fill[black] (-0.063274, -0.419927) circle (0.2pt);
          \fill[black] (0.405804, -0.125166) circle (0.2pt);
          \fill[black] (-0.993712, 0.111964) circle (0.2pt);
          \fill[black] (-0.993712, -0.111964) circle (0.2pt);
          \fill[black] (-0.330279, -0.943883) circle (0.2pt);
          \fill[black] (-0.111964, -0.993712) circle (0.2pt);
          \fill[black] (0.111964, -0.993712) circle (0.2pt);
          \fill[black] (0.943883, -0.330279) circle (0.2pt);
          \fill[black] (0.993712, -0.111964) circle (0.2pt);
          \fill[black] (0.993712, 0.111964) circle (0.2pt);
          \fill[black] (0.330279, 0.943883) circle (0.2pt);
          \fill[black] (0.111964, 0.993712) circle (0.2pt);
          \fill[black] (-0.943883, 0.330279) circle (0.2pt);
          \fill[black] (-0.111964, 0.993712) circle (0.2pt);
          \fill[black] (-0.532032, 0.846724) circle (0.2pt);
          \fill[black] (-0.707107, 0.707107) circle (0.2pt);
          \fill[black] (-0.846724, 0.532032) circle (0.2pt);
          \fill[black] (-0.330279, 0.943883) circle (0.2pt);
          \fill[black] (-0.846724, -0.532032) circle (0.2pt);
          \fill[black] (-0.707107, -0.707107) circle (0.2pt);
          \fill[black] (-0.532032, -0.846724) circle (0.2pt);
          \fill[black] (-0.943883, -0.330279) circle (0.2pt);
          \fill[black] (0.532032, -0.846724) circle (0.2pt);
          \fill[black] (0.707107, -0.707107) circle (0.2pt);
          \fill[black] (0.846724, -0.532032) circle (0.2pt);
          \fill[black] (0.330279, -0.943883) circle (0.2pt);
          \fill[black] (0.846724, 0.532032) circle (0.2pt);
          \fill[black] (0.707107, 0.707107) circle (0.2pt);
          \fill[black] (0.532032, 0.846724) circle (0.2pt);
          \fill[black] (0.943883, 0.330279) circle (0.2pt);

          \coordinate (x0) at (-0.145531, 0.231610);
          \coordinate (x1) at (-0.063274, -0.419927);
          \coordinate (x2) at (0.405804, -0.125166);
          \coordinate (x3) at (-0.993712, 0.111964);
          \coordinate (x4) at (-0.993712, -0.111964);
          \coordinate (x5) at (-0.330279, -0.943883);
          \coordinate (x6) at (-0.111964, -0.993712);
          \coordinate (x7) at (0.111964, -0.993712);
          \coordinate (x8) at (0.943883, -0.330279);
          \coordinate (x9) at (0.993712, -0.111964);
          \coordinate (x10) at (0.993712, 0.111964);
          \coordinate (x11) at (0.330279, 0.943883);
          \coordinate (x12) at (0.111964, 0.993712);
          \coordinate (x13) at (-0.943883, 0.330279);
          \coordinate (x14) at (-0.111964, 0.993712);
          \coordinate (x15) at (-0.532032, 0.846724);
          \coordinate (x16) at (-0.707107, 0.707107);
          \coordinate (x17) at (-0.846724, 0.532032);
          \coordinate (x18) at (-0.330279, 0.943883);
          \coordinate (x19) at (-0.846724, -0.532032);
          \coordinate (x20) at (-0.707107, -0.707107);
          \coordinate (x21) at (-0.532032, -0.846724);
          \coordinate (x22) at (-0.943883, -0.330279);
          \coordinate (x23) at (0.532032, -0.846724);
          \coordinate (x24) at (0.707107, -0.707107);
          \coordinate (x25) at (0.846724, -0.532032);
          \coordinate (x26) at (0.330279, -0.943883);
          \coordinate (x27) at (0.846724, 0.532032);
          \coordinate (x28) at (0.707107, 0.707107);
          \coordinate (x29) at (0.532032, 0.846724);
          \coordinate (x30) at (0.943883, 0.330279);

          \draw (x0) -- (x1);
          \draw (x1) -- (x2);
          \draw (x2) -- (x0);
          \draw (x0) -- (x3);
          \draw (x3) -- (x4);
          \draw[ldiamond] (x4) -- coordinate[midway](m1) (x5);
          \draw (x5) -- (x1);
          \draw (x5) -- (x6);
          \draw (x6) -- (x1);
          \draw (x6) -- (x7);
          \draw[ldiamond] (x7) -- coordinate[midway](m2) (x8);
          \draw (x8) -- (x2);
          \draw (x8) -- (x9);
          \draw (x9) -- (x2);
          \draw (x9) -- (x10);
          \draw[ldiamond] (x10) -- coordinate[midway](m3) (x11);
          \draw (x11) -- (x0);
          \draw (x11) -- (x12);
          \draw (x12) -- (x13);
          \draw (x13) -- (x3);
          \draw (x12) -- (x14);
          \draw (x14) -- (x13);
          \draw (x14) -- (x15);
          \draw (x15) -- (x16);
          \draw (x16) -- (x17);
          \draw (x17) -- (x13);
          \draw (x14) -- (x18);
          \draw (x18) -- (x15);
          \draw (x4) -- (x19);
          \draw (x19) -- (x20);
          \draw (x20) -- (x21);
          \draw (x21) -- (x5);
          \draw (x4) -- (x22);
          \draw (x22) -- (x19);
          \draw (x7) -- (x23);
          \draw (x23) -- (x24);
          \draw (x24) -- (x25);
          \draw (x25) -- (x8);
          \draw (x7) -- (x26);
          \draw (x26) -- (x23);
          \draw (x10) -- (x27);
          \draw (x27) -- (x28);
          \draw (x28) -- (x29);
          \draw (x29) -- (x11);
          \draw (x10) -- (x30);
          \draw (x30) -- (x27);

          \draw[lface] (m1) -- (n2);
          \draw[lface] (m1) -- (n11);
          \draw[lface] (m2) -- (n4);
          \draw[lface] (m2) -- (n13);
          \draw[lface] (m3) -- (n6);
          \draw[lface] (m3) -- (n15);
          \draw[lface] (x13) -- (n7);
          \draw[lface] (x13) -- (n9);

          \node[lvertex] at (x13) {};
        \end{scope}
      \end{tikzpicture}
      {\par\noindent\centering\footnotesize (d) $\mathcal{P}_F$}
    }
    \hspace{10px}
    \parbox{0.45\textwidth}{
      \center
      \begin{tikzpicture}[line cap=round, line join=round]
        \begin{scope}[scale=3]
          \node (n1) at (0.800523, 0.062019) {3};
          \node (n2) at (0.931003, 0.140216) {3};
          \node (n3) at (0.910836, 0.232924) {3};
          \node (n4) at (0.640446, 0.281189) {7};
          \node (n5) at (0.138539, 0.124657) {3};
          \node (n6) at (0.152558, -0.085147) {3};
          \node (n7) at (0.363865, -0.075274) {3};
          \node (n8) at (0.627263, -0.283950) {7};
          \node (n9) at (0.231615, -0.789247) {3};
          \node (n10) at (0.053369, -0.891434) {3};
          \node (n11) at (-0.109343, -0.608463) {7};
          \node (n12) at (-0.511642, -0.595120) {3};
          \node (n13) at (-0.727638, -0.306583) {3};
          \node (n14) at (-0.513635, 0.127724) {7};
          \node[anchor=0] (n15) at (-0.970516, 0.211123) {3};
          \node (n16) at (-0.933787, -0.213847) {3};
          \node (n17) at (0.215101, -0.928496) {3};
          \node[anchor=180] (n18) at (0.970516, -0.211123) {3};
          \node (n19) at (0.306193, 0.726498) {7};
          \node[anchor=240] (n20) at (0.595211, 0.795109) {3};
          \node[anchor=260] (n21) at (0.116488, 0.977285) {3};
          \node[anchor=260] (n22) at (0.211123, 0.970516) {3};
          \node (n23) at (-0.204433, 0.852693) {3};
          \node (n24) at (-0.573205, 0.690563) {7};
          \node[anchor=280] (n25) at (-0.475997, 0.871723) {3};
          \node[anchor=300] (n26) at (-0.795109, 0.595211) {3};
          \node (n27) at (-0.626031, -0.704383) {7};
          \node[anchor= 30] (n28) at (-0.871723, -0.475997) {3};
          \node[anchor= 60] (n29) at (-0.595211, -0.795109) {3};
          \node (n30) at (0.592349, -0.705907) {7};
          \node[anchor=120] (n31) at (0.475997, -0.871723) {3};
          \node[anchor=150] (n32) at (0.795109, -0.595211) {3};

          \fill[black] (0.997452, 0.071339) circle (0.3pt);
          \fill[black] (0.818411, 0.136744) circle (0.3pt);
          \fill[black] (0.585705, -0.022025) circle (0.3pt);
          \fill[black] (0.977147, 0.212565) circle (0.3pt);
          \fill[black] (0.936950, 0.349464) circle (0.3pt);
          \fill[black] (0.877679, 0.479249) circle (0.3pt);
          \fill[black] (0.800541, 0.599278) circle (0.3pt);
          \fill[black] (0.207721, 0.383537) circle (0.3pt);
          \fill[black] (0.256114, 0.042078) circle (0.3pt);
          \fill[black] (-0.048218, -0.051643) circle (0.3pt);
          \fill[black] (0.249777, -0.245876) circle (0.3pt);
          \fill[black] (0.160107, -0.679399) circle (0.3pt);
          \fill[black] (0.463398, -0.690889) circle (0.3pt);
          \fill[black] (0.936950, -0.349464) circle (0.3pt);
          \fill[black] (0.997452, -0.071339) circle (0.3pt);
          \fill[black] (0.071339, -0.997452) circle (0.3pt);
          \fill[black] (-0.071339, -0.997452) circle (0.3pt);
          \fill[black] (-0.493700, -0.370772) circle (0.3pt);
          \fill[black] (-0.349464, -0.936950) circle (0.3pt);
          \fill[black] (-0.212565, -0.977147) circle (0.3pt);
          \fill[black] (-0.691761, -0.477638) circle (0.3pt);
          \fill[black] (-0.997452, -0.071339) circle (0.3pt);
          \fill[black] (-0.329395, 0.583481) circle (0.3pt);
          \fill[black] (-0.936950, 0.349464) circle (0.3pt);
          \fill[black] (-0.997452, 0.071339) circle (0.3pt);
          \fill[black] (-0.977147, 0.212565) circle (0.3pt);
          \fill[black] (-0.977147, -0.212565) circle (0.3pt);
          \fill[black] (0.212565, -0.977147) circle (0.3pt);
          \fill[black] (0.977147, -0.212565) circle (0.3pt);
          \fill[black] (0.707107, 0.707107) circle (0.3pt);
          \fill[black] (0.479249, 0.877679) circle (0.3pt);
          \fill[black] (0.349464, 0.936950) circle (0.3pt);
          \fill[black] (-0.071339, 0.997452) circle (0.3pt);
          \fill[black] (0.599278, 0.800541) circle (0.3pt);
          \fill[black] (0.071339, 0.997452) circle (0.3pt);
          \fill[black] (0.212565, 0.977147) circle (0.3pt);
          \fill[black] (-0.212565, 0.977147) circle (0.3pt);
          \fill[black] (-0.349464, 0.936950) circle (0.3pt);
          \fill[black] (-0.599278, 0.800541) circle (0.3pt);
          \fill[black] (-0.707107, 0.707107) circle (0.3pt);
          \fill[black] (-0.877679, 0.479249) circle (0.3pt);
          \fill[black] (-0.479249, 0.877679) circle (0.3pt);
          \fill[black] (-0.800541, 0.599278) circle (0.3pt);
          \fill[black] (-0.936950, -0.349464) circle (0.3pt);
          \fill[black] (-0.800541, -0.599278) circle (0.3pt);
          \fill[black] (-0.707107, -0.707107) circle (0.3pt);
          \fill[black] (-0.479249, -0.877679) circle (0.3pt);
          \fill[black] (-0.877679, -0.479249) circle (0.3pt);
          \fill[black] (-0.599278, -0.800541) circle (0.3pt);
          \fill[black] (0.349464, -0.936950) circle (0.3pt);
          \fill[black] (0.599278, -0.800541) circle (0.3pt);
          \fill[black] (0.707107, -0.707107) circle (0.3pt);
          \fill[black] (0.877679, -0.479249) circle (0.3pt);
          \fill[black] (0.479249, -0.877679) circle (0.3pt);
          \fill[black] (0.800541, -0.599278) circle (0.3pt);

          \coordinate (x0) at (0.997452, 0.071339);
          \coordinate (x1) at (0.818411, 0.136744);
          \coordinate (x2) at (0.585705, -0.022025);
          \coordinate (x3) at (0.977147, 0.212565);
          \coordinate (x4) at (0.936950, 0.349464);
          \coordinate (x5) at (0.877679, 0.479249);
          \coordinate (x6) at (0.800541, 0.599278);
          \coordinate (x7) at (0.207721, 0.383537);
          \coordinate (x8) at (0.256114, 0.042078);
          \coordinate (x9) at (-0.048218, -0.051643);
          \coordinate (x10) at (0.249777, -0.245876);
          \coordinate (x11) at (0.160107, -0.679399);
          \coordinate (x12) at (0.463398, -0.690889);
          \coordinate (x13) at (0.936950, -0.349464);
          \coordinate (x14) at (0.997452, -0.071339);
          \coordinate (x15) at (0.071339, -0.997452);
          \coordinate (x16) at (-0.071339, -0.997452);
          \coordinate (x17) at (-0.493700, -0.370772);
          \coordinate (x18) at (-0.349464, -0.936950);
          \coordinate (x19) at (-0.212565, -0.977147);
          \coordinate (x20) at (-0.691761, -0.477638);
          \coordinate (x21) at (-0.997452, -0.071339);
          \coordinate (x22) at (-0.329395, 0.583481);
          \coordinate (x23) at (-0.936950, 0.349464);
          \coordinate (x24) at (-0.997452, 0.071339);
          \coordinate (x25) at (-0.977147, 0.212565);
          \coordinate (x26) at (-0.977147, -0.212565);
          \coordinate (x27) at (0.212565, -0.977147);
          \coordinate (x28) at (0.977147, -0.212565);
          \coordinate (x29) at (0.707107, 0.707107);
          \coordinate (x30) at (0.479249, 0.877679);
          \coordinate (x31) at (0.349464, 0.936950);
          \coordinate (x32) at (-0.071339, 0.997452);
          \coordinate (x33) at (0.599278, 0.800541);
          \coordinate (x34) at (0.071339, 0.997452);
          \coordinate (x35) at (0.212565, 0.977147);
          \coordinate (x36) at (-0.212565, 0.977147);
          \coordinate (x37) at (-0.349464, 0.936950);
          \coordinate (x38) at (-0.599278, 0.800541);
          \coordinate (x39) at (-0.707107, 0.707107);
          \coordinate (x40) at (-0.877679, 0.479249);
          \coordinate (x41) at (-0.479249, 0.877679);
          \coordinate (x42) at (-0.800541, 0.599278);
          \coordinate (x43) at (-0.936950, -0.349464);
          \coordinate (x44) at (-0.800541, -0.599278);
          \coordinate (x45) at (-0.707107, -0.707107);
          \coordinate (x46) at (-0.479249, -0.877679);
          \coordinate (x47) at (-0.877679, -0.479249);
          \coordinate (x48) at (-0.599278, -0.800541);
          \coordinate (x49) at (0.349464, -0.936950);
          \coordinate (x50) at (0.599278, -0.800541);
          \coordinate (x51) at (0.707107, -0.707107);
          \coordinate (x52) at (0.877679, -0.479249);
          \coordinate (x53) at (0.479249, -0.877679);
          \coordinate (x54) at (0.800541, -0.599278);

          \draw (x0) -- (x1);
          \draw (x1) -- (x2);
          \draw (x2) -- (x0);
          \draw (x0) -- (x3);
          \draw (x3) -- (x1);
          \draw (x3) -- (x4);
          \draw (x4) -- (x1);
          \draw (x4) -- (x5);
          \draw (x5) -- (x6);
          \draw (x6) -- (x7);
          \draw (x7) -- (x8);
          \draw (x8) -- (x2);
          \draw (x7) -- (x9);
          \draw[lsquare] (x9) -- (x8);
          \draw[lface] (x8) -- (n4);
          \draw[lface] (x9) -- (n11);
          \draw (x9) -- (x10);
          \draw (x10) -- (x8);
          \draw (x10) -- (x2);
          \draw (x10) -- (x11);
          \draw (x11) -- (x12);
          \draw[ldiamond] (x12) -- (x13) node[midway] (m1) {};
          \draw[lface] (m1) -- (n8);
          \draw[lface] (m1) -- (n30);
          \draw (x13) -- (x14);
          \draw (x14) -- (x0);
          \draw (x11) -- (x15);
          \draw (x15) -- (x12);
          \draw (x11) -- (x16);
          \draw (x16) -- (x15);
          \draw (x9) -- (x17);
          \draw (x17) -- (x18);
          \draw (x18) -- (x19);
          \draw (x19) -- (x16);
          \draw (x17) -- (x20);
          \draw (x20) -- (x18);
          \draw (x17) -- (x21);
          \draw[lsquare] (x21) -- (x20);
          \draw[lface] (x21) -- (n14);
          \draw[lface] (x20) -- (n27);
          \draw (x7) -- (x22);
          \draw (x22) -- (x23);
          \draw (x23) -- (x24);
          \draw (x24) -- (x21);
          \draw (x23) -- (x25);
          \draw (x25) -- (x24);
          \draw (x21) -- (x26);
          \draw (x26) -- (x20);
          \draw (x15) -- (x27);
          \draw (x27) -- (x12);
          \draw (x13) -- (x28);
          \draw (x28) -- (x14);
          \draw (x6) -- (x29);
          \draw (x29) -- (x30);
          \draw (x30) -- (x31);
          \draw (x31) -- (x32);
          \draw (x32) -- (x22);
          \draw (x29) -- (x33);
          \draw (x33) -- (x30);
          \draw (x31) -- (x34);
          \draw (x34) -- (x32);
          \draw (x31) -- (x35);
          \draw (x35) -- (x34);
          \draw (x32) -- (x36);
          \draw (x36) -- (x22);
          \draw (x36) -- (x37);
          \draw (x37) -- (x38);
          \draw (x38) -- (x39);
          \draw (x39) -- (x40);
          \draw (x40) -- (x23);
          \draw (x37) -- (x41);
          \draw (x41) -- (x38);
          \draw (x39) -- (x42);
          \draw (x42) -- (x40);
          \draw (x26) -- (x43);
          \draw (x43) -- (x44);
          \draw (x44) -- (x45);
          \draw (x45) -- (x46);
          \draw (x46) -- (x18);
          \draw (x43) -- (x47);
          \draw (x47) -- (x44);
          \draw (x45) -- (x48);
          \draw (x48) -- (x46);
          \draw (x27) -- (x49);
          \draw (x49) -- (x50);
          \draw (x50) -- (x51);
          \draw (x51) -- (x52);
          \draw (x52) -- (x13);
          \draw (x49) -- (x53);
          \draw (x53) -- (x50);
          \draw (x51) -- (x54);
          \draw (x54) -- (x52);

          \node[lvertex] at (x22) {};
          \draw[lface] (x22) -- (n19);
          \draw[lface] (x22) -- (n24);
        \end{scope}
      \end{tikzpicture}
      {\par\noindent\centering\footnotesize (e) $\mathcal{P}_F$}
    }
  \end{center}
  \caption{$\mathcal{P}_N = \mathcal{P}_P$, the edge patch of $\mathcal{P}_N$ and $\mathcal{P}_F$}
  \label{fig:expansion:patch:3:4}
\end{figure}
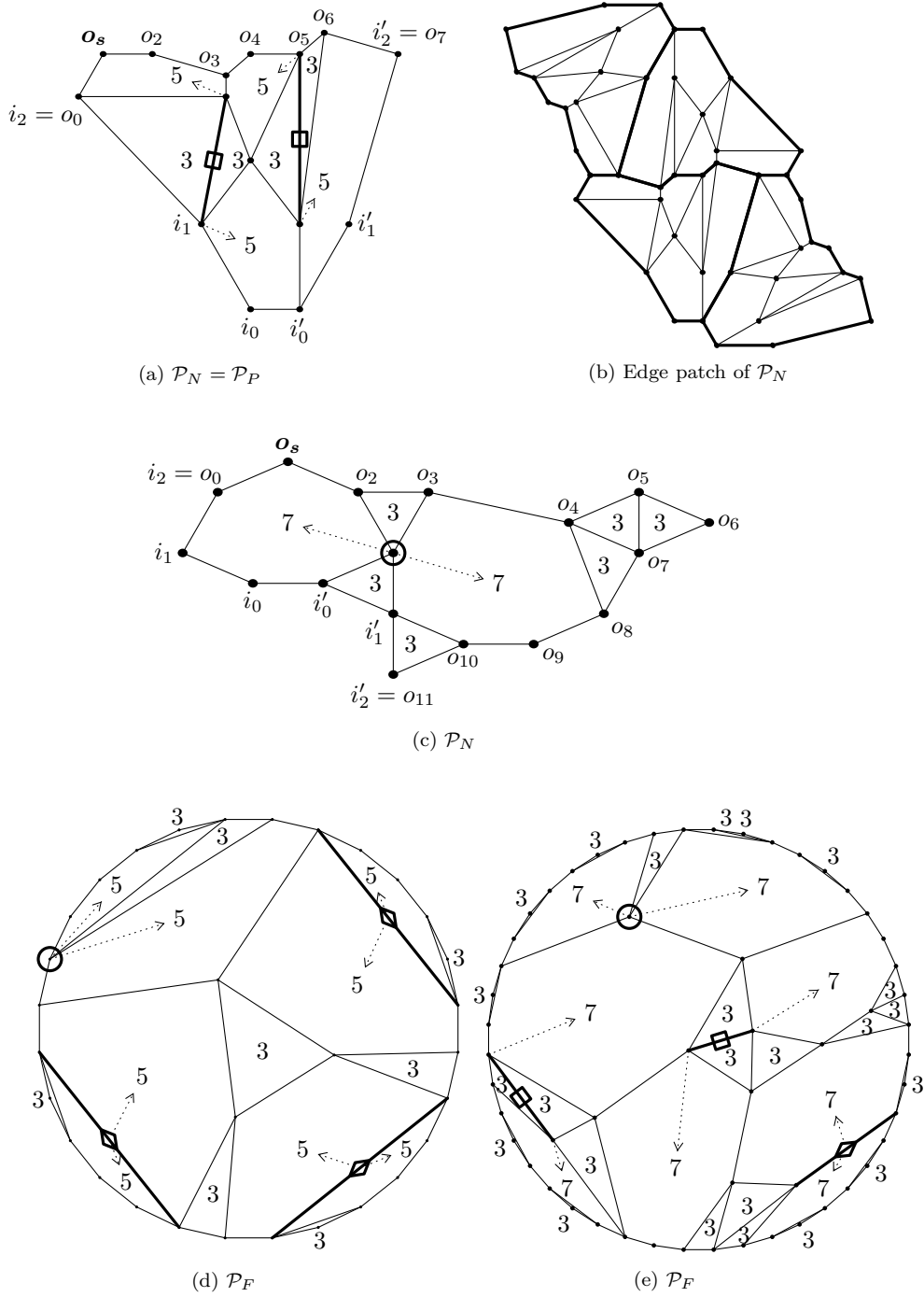

With our whole machinery at work, we can now state the proofs of our
main theorems easily:
\begin{proof}[Proof of Theorem~\ref{eberhard:3:5:4}]
  An expansion $4$-patch $\mathcal{P}_N$ with outer tuple
  $o = (1, 2, 1, 3, \allowbreak 2, 3)$ is shown in
  Fig.~\ref{fig:expansion:patch:3:4}(a) and a corresponding
  $o$-$4$-gonal $4$-patch $\mathcal{P}_F$ is shown in
  Fig.~\ref{fig:expansion:patch:3:4}(d), both consisting of
  triangles and pentagons. By using
  Algorithm~\ref{alg:edge:replacement:3:4:1},
  Algorithm~\ref{alg:edge:replacement:3:4:2} and
  Algorithm~\ref{alg:edge:replacement:3:4:3} as indicated we get
  $4$-patches consisting of only triangles and $(3k+5)$-gons,
  $k \in \mathbb{N}$. We can see in Fig.~\ref{fig:expansion:patch:3:4}(b)
  that $\mathcal{P}_N$ has the polyhedral property, thus we can
  apply Theorem~\ref{thm:main:const} with
  $\mathcal{P}_P \defeq \mathcal{P}_N$.
\end{proof}

\begin{proof}[Proof of Theorem~\ref{eberhard:3:7:4}]
  An expansion $4$-patch $\mathcal{P}_N$ with outer tuple
  $o = (2, 2, 3, 2, 1, \allowbreak 3, 2, 1, 2, 2)$ is shown in
  Fig.~\ref{fig:expansion:patch:3:4}(a) and a corresponding
  $o$-$4$-gonal $4$-patch $\mathcal{P}_F$ is shown in
  Fig.~\ref{fig:expansion:patch:3:4}(c), both of which consist of
  only triangles and heptagons. By using
  Algorithm~\ref{alg:edge:replacement:3:4:1},
  Algorithm~\ref{alg:edge:replacement:3:4:2} and
  Algorithm~\ref{alg:edge:replacement:3:4:3} as indicated we get
  $4$-patches consisting of only triangles and $(3k + 7)$-gons,
  $k \in \mathbb{N}$. We can reuse the $4$-patch $\mathcal{P}_P$ with the
  polyhedral property in Fig.~\ref{fig:expansion:patch:3:4}(a) from
  the last theorem and after application of
  Algorithm~\ref{alg:edge:replacement:3:4:1} it likewise contains
  triangles and $(3k + 7)$-gons. Thus we can apply
  Theorem~\ref{thm:main:const}.
\end{proof}

\end{document}